\documentclass[a4paper,10pt]{amsart}

\usepackage[latin1]{inputenc}
\usepackage{amsmath, amsthm, amssymb}
\usepackage{amscd}
\usepackage[dvips]{graphicx}
\usepackage[all]{xy}
\usepackage{enumerate}
\usepackage{hyperref}

\usepackage{caption}
\usepackage{subcaption}
\usepackage[rightcaption]{sidecap}

\usepackage[usenames,dvipsnames]{pstricks}
\usepackage{epsfig}
\usepackage{pst-grad} 
\usepackage{pst-plot} 
\usepackage{pstricks-add}
\usepackage{pst-solides3d}

\usepackage{color}

\newtheorem{thm}{Theorem}[section]
\newtheorem{cor}[thm]{Corollary}
\newtheorem{prop}[thm]{Proposition}
\newtheorem{defin}[thm]{Definition}

\newtheorem{claim}[thm]{Claim}
\newtheorem*{thm*}{Theorem}
\newtheorem*{prop*}{Proposition}
\newtheorem*{defin*}{Definition}
\newtheorem*{lem*}{Lemma}
\newtheorem*{claim*}{Claim}
\newtheorem*{cor*}{Corollary}

\newtheorem{thmintro}{Theorem}

\theoremstyle{remark}
\newtheorem{rem}[thm]{Remark}
\newtheorem*{rem*}{Remark}

\newcommand\M{\widetilde{M}}

\newcommand{\wt}[1]{\widetilde{#1}}

\newcommand\R{\mathbb R}
\newcommand\Z{\mathbb Z}
\newcommand\N{\mathbb N}

\newcommand\eps{\varepsilon}

\newcommand\flot{ \phi^{t} }
\newcommand\hflot{ \tilde{ \phi}^t }

\newcommand\orb{ \mathcal O }
\newcommand\leafs{ \mathcal L ^{s} }
\newcommand\leafu{ \mathcal L ^{u} }

\newcommand\fs{\mathcal F^{s} }
\newcommand\hfs{\widetilde{\mathcal F}^{s} }
\newcommand\fu{\mathcal F^{u} }
\newcommand\hfu{\widetilde{\mathcal F}^{u} }

\newcommand{\al}[1]{\widetilde{\alpha_{#1} } }

\newcommand\univ{S^1_{\text{univ}}}

\title[Homotopy versus isotopy of closed orbits]{Knot theory of $\R$-covered Anosov flows: homotopy versus isotopy of closed orbits}
\author{Thomas Barthelm\'e}
\address{Tufts University, Medford, MA 02155, USA}
\email{thomas.barthelme@tufts.edu}
\urladdr{sites.google.com/site/thomasbarthelme}

\author{Sergio R.\ Fenley} 
\address{Florida State University, Tallahassee, FL 32306 \and \newline \indent Princeton University, Princeton, NJ 08540, USA}
\email{fenley@math.fsu.edu}

\begin{document}

\begin{abstract}
 In this article, we study the knots realized by periodic orbits of $\R$-covered Anosov flows in compact $3$-manifolds. We show
that if two orbits are freely homotopic then in fact they are 
isotopic. We show that lifts of periodic orbits to the universal cover are unknotted. When the manifold is
atoroidal, we deduce some finer properties regarding the existence of embedded cylinders connecting
two given homotopic orbits.
\end{abstract}

\maketitle

\section{Introduction}

Most examples of Anosov flows on $3$-manifolds have infinitely many homotopically different periodic orbits. When we restrict
our attention to a given free homotopy class however, the picture was thought to be far less interesting. Indeed, classical
examples have a maximum of two periodic orbits per free homotopy class.
But, in \cite{Fen:AFM}, the second author exhibited some Anosov flows on three-manifolds with 
 a striking and totally unexpected behavior:
in these flows, 
each free homotopy class of a periodic orbit contains infinitely many
other orbits. In this paper, we study the isotopy classes
of the periodic orbits of these Anosov flows. 
The flows 
are what is called skewed $\R$-covered Anosov flows on atoroidal $3$-manifolds.

An Anosov flow is {\em $\R$-covered} if the stable (or equivalently the unstable
foliation) lifts to a foliation in the universal cover which has leaf
space homeomorphic to the real numbers $\R$.
It is in addition called {\em skewed} if it is not topologically
conjugate to a suspension Anosov flow.
If the manifold is also atoroidal then
these flows satisfy the homotopic properties of closed orbits described
above. Notice that the contact Anosov flows constructed in \cite{FouHassel:contact_anosov} 
are of this type.

The research for this article was started by the following question asked by P.\ Foulon: a free
homotopy class of periodic orbits of a skewed $\R$-covered Anosov flow in an
atoroidal manifold gives an infinite
family of homotopically equivalent embeddings of $S^1$ in a $3$-manifold $M$, i.e., knots in $M$. 
One very natural question is whether
these knots are different in the sense of knot theory?
That is, are they isotopic curves or not?
The question of studying knots associated with some dynamical system is not new (see for instance the survey by E.\ Ghys \cite{Ghys:knots_and_dynamics}), the difference here being that the manifold $M$ has (a rich) topology.

Our main result is the following:

\begin{thmintro}
 Let $\flot$ be a skewed $\R$-covered Anosov flow on a closed $3$-manifold. 
Suppose that the stable foliation
is transversely orientable.
If two periodic orbits of $\flot$ are freely homotopic, then they are isotopic.
\end{thmintro}

This result shows that the periodic orbits of a skewed $\R$-covered Anosov flow 
in a given homotopy class are surprisingly similar!
We do not need to assume that the manifold 
is atoroidal for this result. When it is atoroidal, we have the
\begin{cor*}
 Let $\flot$ be a skewed $\R$-covered Anosov flow on an \emph{atoroidal} closed $3$-manifold.
 Every periodic orbit is \emph{isotopic} to infinitely many other closed orbits.
\end{cor*}


To study this question, we start by using several objects associated 
with the stable foliation of the flow. Namely, we use the universal circle of the 
stable foliation, as well as certain geometric walls obtained using the universal circle.
We finish by a 
very careful analysis of the possible self intersections of an a priori
only immersed annulus realizing the homotopy between the closed orbits.
This annulus is obtained as the quotient of the walls mentioned above.
This result shows that $M - \alpha_1$ is homeomorphic to 
$M - \alpha_2$, by a homeomorphism induced from $M$ which
is isotopic to the identity.

\vskip .1in
The next natural question is whether any such 
two isotopic orbits could be linked by an \emph{embedded} annulus. 
This question can be resolved in the atoroidal
case using a regulating pseudo-Anosov flow and has deep connections
with the action of the fundamental group on the orbit 
space of the flow and the universal circle of 
the foliation. We prove:
\begin{thmintro}
Let $\phi^t$ be a skewed, $\R$-covered Anosov flow whose stable foliation is transversely orientable.
Suppose that the manifold is atoroidal.
Then, for any closed orbit $\alpha$ of $\phi^t$, there
is a closed orbit $\beta$ which is isotopic to $\alpha$, but
such that there is no embedded annulus in $M$ with boundary
the union of $\alpha$ and $\beta$.
Moreover, for any closed orbit of $\phi^t$, there is \emph{at most} a finite number of orbits that can be the other boundary of an embedded annulus. In fact, there exists closed orbits of $\phi^t$ such that \emph{no} orbits can be the other boundary of an embedded annulus.
\end{thmintro}

To study this question we use a flow obtained from the geometry 
of the Anosov foliation. When the manifold is atoroidal (or hyperbolic), 
Thurston constructed a pseudo-Anosov flow which is transverse to the
stable foliation. 
This helps us establish the result. We also obtain some finer
properties about embedded annuli.
Notice that Theorem B is false for \emph{toroidal} manifold as can be easily seen by
considering
the geodesic flow of an hyperbolic surface.

Finally we prove that if $\gamma$ is a closed orbit of an
Anosov or pseudo-Anosov flow and $\wt{\gamma}$ is a lift to
the universal cover of the manifold, then $\wt{\gamma}$ is
unknotted in the universal cover.

\subsection*{Acknowledgements}
The first author would like to thank Patrick Foulon and Boris Hasselblatt for introducing him to this question and for the numerous and fruitful discussions afterwards.
Both authors thank Boris Hasselblatt for 
useful suggestions to an earlier version of this article.

\section{Background}

\subsection{Anosov flows}
Recall that an Anosov flow is defined as follows:

\begin{defin} \label{def:Anosov}
 Let $M$ be a compact manifold and $\flot \colon M \rightarrow M$ a $C^{1}$ flow on $M$. The flow $\flot$ is called Anosov if there exists a splitting of the tangent bundle ${TM =  \R\cdot X \oplus E^{ss} \oplus E^{uu}}$ preserved by $D\flot$ and two constants $a,b >0$ such that:
\begin{enumerate}
 \item $X$ is the generating vector field of $\flot$;
 \item For any $v\in E^{ss}$ and $t>0$,
    \begin{equation*}
     \lVert D\flot(v)\rVert \leq be^{-at}\lVert v \rVert \, ;
    \end{equation*}
  \item For any $v\in E^{uu}$ and $t>0$,
    \begin{equation*}
     \lVert D\phi^{-t}(v)\rVert \leq be^{-at}\lVert v \rVert\, .
    \end{equation*}
\end{enumerate}
In the above, $\lVert \cdot \rVert$ is any Riemannian (or Finsler) metric on $M$.
\end{defin}

The subbundle $E^{ss}$ (resp. $E^{uu}$) is called the \emph{strong stable distribution} (resp.
\emph{strong unstable distribution}). It is a classical result of 
Anosov (\cite{Anosov}) that $E^{ss}$, $E^{uu}$, $\R\cdot X \oplus E^{ss}$ and $\R\cdot X \oplus E^{uu}$ are integrable. We denote by $\mathcal{F}^{ss}$, $\mathcal{F}^{uu}$, $\fs$ and $\fu$ the respective foliations and we call them the strong stable, strong unstable, stable and unstable foliations.

In all the following, if $x \in M$, then $\fs(x)$ (resp. $\fu(x)$) is the leaf of the foliation $\fs$ (resp. $\fu$) containing $x$.

\subsection{Leaf and orbit spaces}
Our study of the Anosov flow will be through its foliations. 
Let $\M$ be the universal cover of $M$ and $\pi \colon \M \rightarrow M$ the canonical projection. The flow $\flot$ and all the foliations lift to $\M$ and we denote them respectively by $\hflot$, $\widetilde{\mathcal{F}}^{ss} $, $\hfs$, $\widetilde{\mathcal{F}}^{uu}$ and $\hfu$. Now we can define
\begin{itemize}
 \item The \emph{orbit space} of $\flot$ as $\M$ quotiented out by the relation ``being on the
same orbit of $\hflot$ ''. We denote it by $\orb$.
 \item The \emph{stable} (resp. \emph{unstable}) \emph{leaf space} of $\flot$ as $\M$ quotiented out by the relation ``being on the same leaf of $\hfs$ (resp. $\hfu$)''. We denote them by $\leafs$ and $\leafu$ respectively.
\end{itemize}
Note that the foliations $\hfs$ and $\hfu$ obviously project to two transverse 
one dimensional foliations of $\orb$. We will keep the same notations for the projected foliations, hoping that it will not lead to any confusion.

Thierry Barbot and the second author started the study of Anosov flows through their orbit
spaces, in addition to analyzing the leaf spaces.
They proved that the orbit space is always homeomorphic to $\R^2$ \cite{Bar:CFA,Fen:AFM}. 
Prior to that, both Verjovsky \cite{Ver:codim1} and Ghys \cite{Ghys:varietes_fibrees_en_cercles} had used the study
of the leaf spaces to obtain important results about Anosov flows.

\subsection{$\R$-covered Anosov flows}
If one considers the topology of the leaf spaces, differences between flows start to appear. In this article, we are only interested in one special kind:

\begin{defin}[Barbot \cite{Bar:these}, Fenley \cite{Fen:AFM}]
 An Anosov flow is said to be $\R$-covered if $\leafs$, or equivalently $\leafu$, is homeomorphic to $\R$.
\end{defin}

For the results of this article
we will furthermore assume that the stable foliation is \emph{transversely oriented}, i.e., 
the foliation $\mathcal{F}^{uu}$ is orientable.
In the case of skewed, $\R$-covered Anosov flows which will be studied 
in this paper, this is equivalent to the unstable foliation (the
other foliation) being transversely oriented. In particular, any orientation on the stable and
unstable foliations gives an orientation on $\leafs$ and $\leafu$. We will 
assume that such an orientation is chosen.

In terms of the structure of the stable and unstable foliations in the
orbit space
there are only two types of $\R$-covered Anosov flows. 
This depends on whether or not there exists a stable leaf 
of the lifted flow $\hflot$ intersecting every unstable leaf.
 In addition, Barbot \cite{Bar:CFA} proved the very strong
result that either the flow is topologically conjugate to a suspension of an Anosov
diffeomorphism, or the orbit space $\orb$ is homeomorphic to a diagonal strip $|x-y|<1$ in $\R^2$
where the unstable leaves are given by $x=\text{const}$ and the stable leaves are given by
$y=\text{const}$ (see Figure \ref{fig:lozenge_def_by_orbit}). This second type of $\R$-covered
Anosov flows is called \emph{skewed} and these flows are the object of our study.

\subsection{Lozenges}

A fundamental object in the study of (pseudo-)Anosov flows is a 
lozenge. First we define a \emph{half leaf} of a leaf $H$  of $\hfs$ or $\hfu$ as a component of
$H - \gamma$, where $\gamma$ is any orbit in $H$.

\begin{defin}
 A lozenge $L$ in $\orb$ is an open subset of $\orb$ such that (see Figure \ref{fig:a_lozenge}):\\
There exist two points $\alpha,\beta \in \orb$ and four half leaves $A \subset \hfs(\alpha)$, $B \subset \hfu(\alpha)$, $C \subset \hfs(\beta)$ and $D \subset \hfu(\beta)$ satisfying:
\begin{itemize}
 \item For any $\lambda^s \in \leafs$, $\lambda^s \cap B \neq \emptyset$ if and only if $\lambda^s \cap D\neq \emptyset$,
 \item For any $\lambda^u \in \leafu$, $\lambda^u \cap A \neq \emptyset$ if and only if $\lambda^u \cap C \neq \emptyset$,
 \item The half-leaf $A$ does not intersect $D$ and $B$ does not intersect $C$.
\end{itemize}
Then,
\begin{equation*}
 L := \lbrace p \in \orb \mid \hfs(p) \cap B \neq \emptyset, \; \hfu(p) \cap A \neq \emptyset \rbrace.
\end{equation*}
The points $\alpha$ and $\beta$ are called the \emph{corners} of $L$ and $A,B,C$ and $D$ are called the \emph{sides}.
\end{defin}

\begin{figure}[h]
\begin{center}

\begin{pspicture}(0,-1.97)(3.92,1.97)
\psbezier[linewidth=0.04](0.02,0.17)(0.88,0.03)(0.74114233,-0.47874346)(1.48,-1.11)(2.2188578,-1.7412565)(2.38,-1.73)(3.26,-1.95)
\psbezier[linewidth=0.04](0.0,0.27)(0.96,0.3105634)(0.75286174,0.37057108)(1.78,0.77)(2.8071382,1.169429)(2.66,1.47)(3.36,1.45)
\psbezier[linewidth=0.04](0.6,1.95)(1.2539726,1.8871263)(1.1265805,1.3646309)(2.0345206,0.7973154)(2.9424605,0.23)(3.2249315,0.2543258)(3.9,0.31)
\psbezier[linewidth=0.04](0.52,-1.33)(1.48,-1.33)(1.3597014,-0.9703507)(2.3,-0.63)(3.2402985,-0.28964934)(3.14,0.05)(3.84,0.23)
\psdots[dotsize=0.16](1.98,0.85)
\psdots[dotsize=0.16](1.46,-1.11)
\usefont{T1}{ptm}{m}{n}
\rput(1.6145313,-0.06){$L$}
\usefont{T1}{ptm}{m}{n}
\rput(1.4,-1.38){$\alpha$}
\rput(2,1.2){$\beta$}

\rput(0.6,-0.6){$A$}
\rput(3,-0.6){$B$}
\rput(0.6,0.6){$D$}
\rput(3,0.6){$C$}
\end{pspicture} 
\end{center}
\caption{A lozenge with corners $\alpha$, $\beta$ and sides $A,B,C,D$} \label{fig:a_lozenge}
\end{figure}
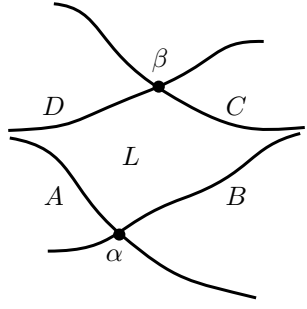

One of the most important properties of lozenges is the following:
\begin{prop}[Fenley \cite{Fen:AFM}] \label{prop:stabilized_lozenge}
 Let $L$ be a lozenge of $\flot$. If one corner of $L$ is stabilized by an element $g$
 of the fundamental group and so are the two sides of $L$ abutting at this corner, 
then $g$ stabilizes $L$, the other corner and the other sides.
\end{prop}

In the case of a skewed, $\R$-covered Anosov flow with transversely orientable
stable foliation, the hypothesis about the sides
of the lozenge is always satisfied as soon a $g$ leaves
invariant a corner of a lozenge $L$.

\begin{defin}
 A chain of lozenges is a union (finite or infinite) of lozenges $L_i$ such that two consecutive lozenges $L_i$ and $L_{i+1}$ always share a corner.
\end{defin}

There are basically two configurations for consecutive lozenges in a chain: either they share a side,
or they do not.
The first case is characterized by the fact that there exists a leaf intersecting the
interior of both lozenges, while it cannot happen in the second case because a stable leaf in $\M$
cannot intersects an unstable leaf in more than one orbit (\cite{Ver:codim1}, 
explicit proof in \cite{Fen:AFM} or \cite{Fen:Foliations_TG3M}). 
(See Figure \ref{fig:chain_of_lozenges}).
\begin{figure}[h]
\begin{subfigure}[b]{0.45\textwidth}
\centering
  \scalebox{0.8} { 
\begin{pspicture}(-0.4,-3.8)(4.8,4)
\psdots[dotsize=0.16](2.44,2.6)
\psbezier[linewidth=0.04](0.38,0.18)(1.24,0.04)(1.1011424,-0.46874347)(1.84,-1.1)(2.5788577,-1.7312565)(2.74,-1.72)(3.62,-1.94)
\psbezier[linewidth=0.04](0.36,0.28)(1.32,0.32056338)(1.1128618,0.38057107)(2.14,0.78)(3.1671383,1.1794289)(3.56,1.48)(4.26,1.46)
\psbezier[linewidth=0.04](0.96,1.96)(1.6139725,1.8971263)(1.4865805,1.3746309)(2.3945205,0.80731547)(3.3024607,0.24)(3.5849316,0.26432583)(4.26,0.32)
\psbezier[linewidth=0.04](0.24,-1.64)(1.16,-1.62)(1.7197014,-0.9603507)(2.66,-0.62)(3.6002986,-0.27964935)(3.5,0.06)(4.2,0.24)
\psdots[dotsize=0.16](2.34,0.86)
\psdots[dotsize=0.16](1.82,-1.06)
\psbezier[linewidth=0.04](0.98,3.46)(1.5739726,3.4171262)(1.5465806,3.1746309)(2.4545205,2.6073155)(3.3624606,2.04)(3.64,1.6)(4.22,1.58)
\psbezier[linewidth=0.04](1.0,2.06)(1.96,2.06)(1.8397014,2.4196494)(2.78,2.76)(3.7202985,3.1003506)(3.62,3.44)(4.32,3.62)
\psbezier[linewidth=0.04](0.22,-1.74)(0.8139726,-1.7828737)(0.7865805,-2.0253692)(1.6945206,-2.5926845)(2.6024606,-3.16)(2.88,-3.6)(3.46,-3.62)
\psbezier[linewidth=0.04](0.0,-3.54)(0.94,-3.4394367)(0.54,-3.22)(1.44,-2.74)(2.34,-2.26)(2.86,-2.04)(3.56,-2.06)
\psdots[dotsize=0.16](1.72,-2.62)
\psbezier[linewidth=0.04,linestyle=dashed,dash=0.16cm 0.16cm](0.06,-2.94)(0.54,-2.9)(0.8125543,-2.2401693)(1.7,-1.84)(2.5874457,-1.4398307)(3.08,-0.86)(4.04,-1.06)
\end{pspicture} }
  \caption{Lozenges sharing only corners}
\label{fig:string_lozenges} 
\end{subfigure}
\quad
\begin{subfigure}[b]{0.45\textwidth}
\centering
\scalebox{0.8}{
\begin{pspicture}(6.6,-3.6)(13.8,2)
\psbezier[linewidth=0.04](7.0,-0.32)(7.54,-0.72)(7.6411424,-0.56874347)(8.38,-1.2)(9.118857,-1.8312565)(9.64,-1.6)(10.18,-2.08)
\psbezier[linewidth=0.04](6.78,-1.74)(7.7,-1.72)(7.8397017,-1.2803507)(8.86,-0.86)(9.880299,-0.43964934)(9.72,-0.42)(9.72,0.06)
\psdots[dotsize=0.16](8.32,-1.16)
\psbezier[linewidth=0.04](6.78,-1.76)(7.3739724,-1.8028737)(7.3665805,-1.9853691)(8.25452,-2.6126845)(9.142461,-3.24)(8.6,-2.84)(9.0,-3.14)
\psbezier[linewidth=0.04](7.48,-3.22)(8.32,-2.74)(7.56,-3.12)(8.48,-2.62)(9.4,-2.12)(9.46,-1.82)(10.1,-2.18)
\psdots[dotsize=0.16](8.38,-2.7)
\psbezier[linewidth=0.04](9.84,0.16)(9.92,-0.4)(10.641142,-0.16874346)(11.4,-0.8)(12.158857,-1.4312565)(12.24,-1.44)(13.14,-1.64)
\psbezier[linewidth=0.04](10.26,-1.98)(9.98,-1.5)(10.66,-1.44)(11.56,-0.96)(12.46,-0.48)(12.98,-0.26)(13.68,-0.28)
\psbezier[linewidth=0.04](10.46,1.56)(10.52,0.96)(11.121142,1.3112565)(11.88,0.68)(12.638858,0.048743468)(12.72,0.04)(13.62,-0.16)
\psbezier[linewidth=0.04](9.98,0.22)(10.62,-0.18)(10.26,0.28)(11.12,0.78)(11.98,1.28)(11.5,0.96)(12.14,1.4)
\psdots[dotsize=0.16](11.43,0.95)
\psdots[dotsize=0.16](11.58,-0.95)
\psbezier[linewidth=0.04,linestyle=dashed,dash=0.16cm 0.16cm](6.84,-2.72)(7.24,-2.7)(7.572743,-2.4578607)(8.44,-1.98)(9.307257,-1.5021392)(8.894286,-1.4042312)(9.74,-0.92)(10.585714,-0.4357688)(10.980949,-0.86126333)(11.76,-0.28)(12.539051,0.30126333)(12.3,0.7)(12.76,0.9)
\end{pspicture} }
\caption{Lozenges sharing sides} 
\label{fig:adjacent_lozenges}
\end{subfigure}
\caption{The two types of consecutive lozenges in a chain} \label{fig:chain_of_lozenges}
\end{figure}
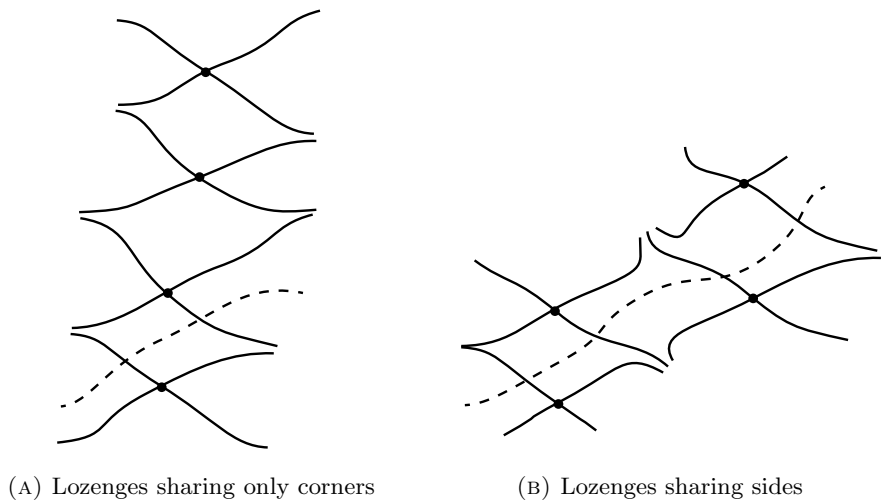

Chains of lozenges in $\R$-covered Anosov flow are particularly nice, because they never share sides:
\begin{prop}[Fenley \cite{Fen:AFM}]
 Let $\flot$ be a skewed $\R$-covered Anosov flow, and $C = \bigcup L_i$ a chain of lozenges. Let $p_{i}$ be the shared corner between $L_i$ and $L_{i+1}$.
Then the union of $p_i$ 
and the sides through $p_i$ of $L_i$ and $L_{i+1}$ is $\hfs(p_i) \cup \hfu(p_i)$. In other words, consecutive lozenges never share a side. In particular, an (un)stable leaf cannot intersect the interior of more than one lozenge in $C$.
\end{prop}

Suppose that 
$C = \bigcup L_i$ is a chain of lozenges such that one of its corner $\al 0$ is a lift of a
periodic orbit, so that there exists $\gamma \in \pi_1(M)$ stabilizing $\al 0$.
By Proposition \ref{prop:stabilized_lozenge}, 
$\gamma$ stabilizes $C$ and therefore all its corners are lifts of (possibly the same) periodic
orbit. A representation on how $\gamma $ acts on the chain of lozenges is given in Figure \ref{fig:gamma_stabilizing_chain}.
\begin{figure}[h]
\begin{center}
\scalebox{0.8}{
\begin{pspicture}(0,0)(6,6)
\rput(2.51,0.4){$\al i$}
\rput(2.51,2.05){$\al{i+1}$}
\rput(2.51,3.7){$\al{i+2}$}
\psset{arrowscale=2}
\psbezier[linewidth=0.04,ArrowInside=->,ArrowInsidePos=0.30](0,0)(2,0)(3,1.4)(5,1.6)
\psbezier[linewidth=0.04,ArrowInside=-<,ArrowInsidePos=0.75](0,0)(2,0)(3,1.4)(5,1.6)

\psbezier[linewidth=0.04,ArrowInside=->,ArrowInsidePos=0.75](5,0)(3,0)(2,1.4)(0,1.6)
\psbezier[linewidth=0.04,ArrowInside=-<,ArrowInsidePos=0.30](5,0)(3,0)(2,1.4)(0,1.6)

\psdots[dotsize=0.2](2.51,0.73)

\rput(0,1.65){
\psbezier[linewidth=0.04,ArrowInside=-<,ArrowInsidePos=0.75](5,0)(3,0)(2,1.4)(0,1.6)
\psbezier[linewidth=0.04,ArrowInside=->,ArrowInsidePos=0.30](5,0)(3,0)(2,1.4)(0,1.6)

\psdots[dotsize=0.2](2.51,0.73)

\psbezier[linewidth=0.04,ArrowInside=-<,ArrowInsidePos=0.30](0,0)(2,0)(3,1.4)(5,1.6)
\psbezier[linewidth=0.04,ArrowInside=->,ArrowInsidePos=0.75](0,0)(2,0)(3,1.4)(5,1.6)
}

\rput(0,3.3){
\psbezier[linewidth=0.04,ArrowInside=->,ArrowInsidePos=0.75](5,0)(3,0)(2,1.4)(0,1.6)
\psbezier[linewidth=0.04,ArrowInside=-<,ArrowInsidePos=0.30](5,0)(3,0)(2,1.4)(0,1.6)

\psdots[dotsize=0.2](2.51,0.73)

\psbezier[linewidth=0.04,ArrowInside=->,ArrowInsidePos=0.30](0,0)(2,0)(3,1.4)(5,1.6)
\psbezier[linewidth=0.04,ArrowInside=-<,ArrowInsidePos=0.75](0,0)(2,0)(3,1.4)(5,1.6)
}
\end{pspicture}}
\end{center}
\caption{The action of an element $\gamma \in \pi_1(M)$ stabilizing a chain of lozenges} \label{fig:gamma_stabilizing_chain}
\end{figure}
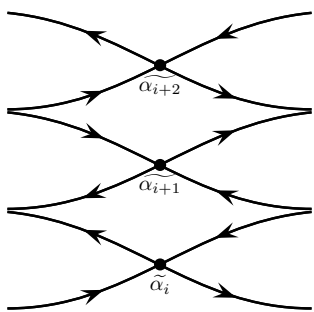

\begin{rem} 
Looking at the orientation of the sides of a lozenge, we can see that they come in two different types.\\
As $\fs$ is assumed transversely orientable, we can chose an orientation on each leaf of $\hfs$
and $\hfu$ when projected
to $\orb$. So any orbit defines two stable half leafs (positive and negative) and two unstable
half leafs. Let now $p$ be a corner of a lozenge $L$. The sides of $L$ going through $p$ --- call them $A$ for the stable and $B$ for the unstable --- could either be both positive, both negative, or of different signs. It is quite easy to see that the stable (resp. unstable) side of the other corner needs to have switched sign from $B$ (resp. from $A$). So each lozenge could be of two types, either $(+,+,-,-)$ or $(+,-,-,+)$, but evidently, all the lozenges of the same (transversally orientable) flow are of the same type (\cite{Fen:AFM}).
\end{rem}
By changing the transverse orientation of one of stable or unstable foliation we may assume without loss of generality that lozenges are of type $(+,+,-,-)$ (see Figure \ref{fig:lozenge++--}).

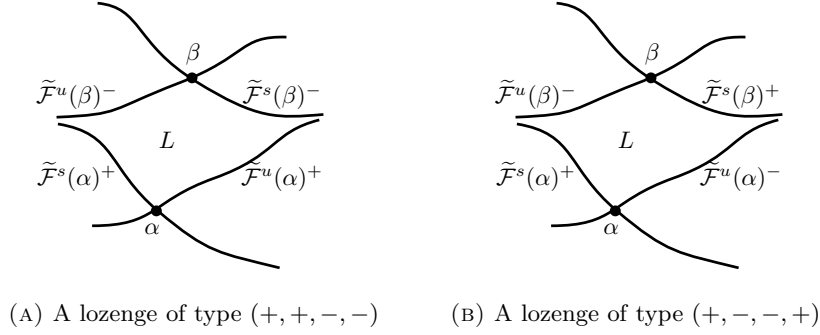
\begin{figure}[h]
\begin{subfigure}[b]{0.45\textwidth}
\centering
 \scalebox{0.9} {
\begin{pspicture}(-1,-2.2)(5,2)
\psbezier[linewidth=0.04](0.02,0.17)(0.88,0.03)(0.74114233,-0.47874346)(1.48,-1.11)(2.2188578,-1.7412565)(2.38,-1.73)(3.26,-1.95)
\psbezier[linewidth=0.04](0.0,0.27)(0.96,0.3105634)(0.75286174,0.37057108)(1.78,0.77)(2.8071382,1.169429)(2.66,1.47)(3.36,1.45)
\psbezier[linewidth=0.04](0.6,1.95)(1.2539726,1.8871263)(1.1265805,1.3646309)(2.0345206,0.7973154)(2.9424605,0.23)(3.2249315,0.2543258)(3.9,0.31)
\psbezier[linewidth=0.04](0.52,-1.33)(1.48,-1.33)(1.3597014,-0.9703507)(2.3,-0.63)(3.2402985,-0.28964934)(3.14,0.05)(3.84,0.23)
\psdots[dotsize=0.16](1.98,0.85)
\psdots[dotsize=0.16](1.46,-1.11)
\usefont{T1}{ptm}{m}{n}
\rput(1.6145313,-0.06){$L$}
\usefont{T1}{ptm}{m}{n}
\rput(1.4,-1.38){$\alpha$}
\rput(2,1.2){$\beta$}

\rput(0.3,-0.6){$\hfs(\alpha)^+$}
\rput(3.3,-0.6){$\hfu(\alpha)^+ $}
\rput(0.3,0.6){$\hfu(\beta)^-$}
\rput(3.3,0.6){$\hfs(\beta)^-$}
\end{pspicture} 
}
 \caption{A lozenge of type $(+,+,-,-)$}
\label{fig:lozenge++--} 
\end{subfigure}
\begin{subfigure}[b]{0.45\textwidth}
 \centering
\scalebox{0.9} {
\begin{pspicture}(-1,-2.2)(4.5,2)
\psbezier[linewidth=0.04](0.02,0.17)(0.88,0.03)(0.74114233,-0.47874346)(1.48,-1.11)(2.2188578,-1.7412565)(2.38,-1.73)(3.26,-1.95)
\psbezier[linewidth=0.04](0.0,0.27)(0.96,0.3105634)(0.75286174,0.37057108)(1.78,0.77)(2.8071382,1.169429)(2.66,1.47)(3.36,1.45)
\psbezier[linewidth=0.04](0.6,1.95)(1.2539726,1.8871263)(1.1265805,1.3646309)(2.0345206,0.7973154)(2.9424605,0.23)(3.2249315,0.2543258)(3.9,0.31)
\psbezier[linewidth=0.04](0.52,-1.33)(1.48,-1.33)(1.3597014,-0.9703507)(2.3,-0.63)(3.2402985,-0.28964934)(3.14,0.05)(3.84,0.23)
\psdots[dotsize=0.16](1.98,0.85)
\psdots[dotsize=0.16](1.46,-1.11)
\usefont{T1}{ptm}{m}{n}
\rput(1.6145313,-0.06){$L$}
\usefont{T1}{ptm}{m}{n}
\rput(1.4,-1.38){$\alpha$}
\rput(2,1.2){$\beta$}

\rput(0.3,-0.6){$\hfs(\alpha)^+$}
\rput(3.3,-0.6){$\hfu(\alpha)^- $}
\rput(0.3,0.6){$\hfu(\beta)^-$}
\rput(3.3,0.6){$\hfs(\beta)^+$}
\end{pspicture} 
}
\caption{A lozenge of type $(+,-,-,+)$}
\label{fig:lozenge+-+-} 
\end{subfigure}
 \caption{The two possible orientations of lozenges} \label{fig:type_of_lozenge}
\end{figure}

%
%

\subsection{Free homotopy classes of periodic orbits}

One of the reasons why
lozenges are particularly important in the study of $\R$-covered Anosov flows is that they are very common:
\begin{prop}[Fenley \cite{Fen:AFM}]
 Let $\flot$ be a skewed, $\R$-covered Anosov flow. Then, any orbit is a corner of two distinct lozenges.
\end{prop}
Let us give an idea of how such lozenges are constructed. Let $\lambda^s \in \leafs$. Then the set $I^u(\lambda^s):= \lbrace \lambda^u \in \leafu \mid \lambda^u \cap \lambda^s \neq \emptyset\rbrace$ is an open, non-empty, connected and bounded set in $\leafu \simeq \R$. Hence it admits an upper and lower bound. Let $\eta^s(\lambda^s) \in \leafu$ be the upper bound and $\eta^{-u}(\lambda^s)\in \leafu$ be the lower bound. Similarly, for any $\lambda^u \in \leafu$, define $\eta^u(\lambda^u)$ and $\eta^{-s}(\lambda^u)$ as, respectively, the upper and lower bounds in $\leafs$ of the set of stable leafs that intersects $\lambda^u$. We have the following result:
\begin{prop}[Fenley \cite{Fen:AFM}, Barbot \cite{Bar:CFA,Bar:PAG}] \label{prop:eta_s_eta_u}
 Let $\flot$ be a skewed $\R$-covered Anosov flow in a $3$-manifold $M$,
where $\fs$ is transversely orientable. Then, the functions $\eta^s \colon \leafs \rightarrow \leafu$ and $\eta^u \colon \leafu \rightarrow \leafs$ are H\"older-homeomorphisms and $\pi_1(M)$-equivariant. We have $(\eta^u)^{-1} = \eta^{-u}$, and $(\eta^s)^{-1} = \eta^{-s}$.
 Furthermore, $\eta^u \circ \eta^s$ and $\eta^s \circ \eta^u$ are strictly increasing homeomorphisms and we can define $\eta \colon \orb  \rightarrow \orb$ by
\begin{equation*}
\eta(o):= \eta^u \left( \hfu(o)\right) \cap \eta^s\left(\hfs(o) \right) .
\end{equation*}
\end{prop}

So, for any orbit $o \in \orb$, the leaves $\hfs(o)$, $\hfu(o)$, $\eta^s(\hfs(o))$ and $\eta^u(\hfu(o))$ contain the side of a unique lozenge of corners $o$ and $\eta(o)$ (see Figure \ref{fig:lozenge_def_by_orbit}).

\begin{figure}[h]
\begin{center}
\begin{pspicture}(-0.5,-0.5)(6,6)
\psline[linewidth=0.04cm,linestyle=dashed](2,0.5)(5,3.5)
\psline[linewidth=0.04cm,linestyle=dashed](0.5,2)(3.5,5)
\psline[linewidth=0.04cm,arrowsize=0.05cm 2.0,arrowlength=1.4,arrowinset=0.4]{->}(1,0)(5,0)
\psline[linewidth=0.04cm,arrowsize=0.05cm 2.0,arrowlength=1.4,arrowinset=0.4]{->}(0,1)(0,5)
\rput(5.2,-0.2){$\leafs$}
\rput(-0.2,5.2){$\leafu$}
 \psline[linewidth=0.04cm,linecolor=green](0.5,2)(3.5,2)
\rput(4.1,1.8){$\hfu(o)$}
 \psline[linewidth=0.04cm,linecolor=red](2,0.5)(2,3.5)
\rput(2,0.2){$\hfs(o)$}
\put(1.6,1.6){$o$}
\psline[linewidth=0.04cm,linecolor=red,linestyle=dashed](3.5,2)(3.5,5)
 \rput(3.5,5.3){$\eta^u\left(\hfu(o)\right)$}
\psline[linewidth=0.04cm,linecolor=green,linestyle=dashed](2,3.5)(5,3.5)
\rput(1.1,3.9){$\eta^s\left(\hfs(o)\right)$}
\put(3.6,3.6){$\eta\left(o\right)$ }
\put(2.6,2.6){$L$}
\end{pspicture} 
\end{center}
 \caption{A lozenge $L$ defined by the orbit $o$} \label{fig:lozenge_def_by_orbit}
\end{figure}
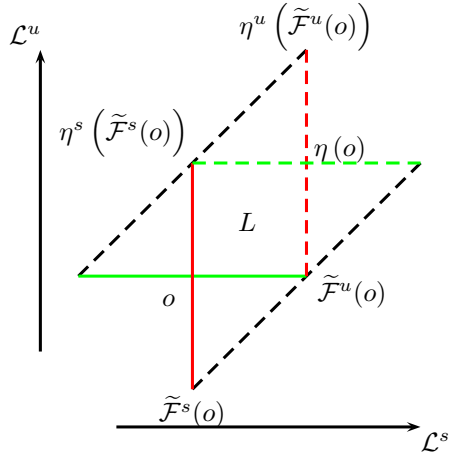

Now, thanks to Proposition \ref{prop:eta_s_eta_u}, we can describe explicitly the free homotopy class of a periodic orbit of an $\R$-covered Anosov flow:
\begin{thm}[Fenley \cite{Fen:AFM}] \label{thm:infinite_homotopy_class}
 Let $\flot$ be a skewed $\R$-covered Anosov flow on a homotopically atoroidal
$3$-manifold $M$,
so that $\fs$ is transversely orientable.
Let
$\alpha$ be a periodic orbit of $\flot$ and $\wt{\alpha}$ a lift of $\alpha$ to the universal cover $\M$.
The family of orbits $\lbrace \eta^{i}\left( \wt{\alpha}\right) \rbrace_{i\in \Z}$ projects down to 
a family of distinct periodic orbits. Furthermore, for any $i\in \Z$, the projection to $M$ of the
orbit $\eta^{2i}\left( \wt{\alpha} \right)$ is a periodic orbit freely homotopic to $\alpha$, and the
projection to $M$ of the orbit $\eta^{2i+1}\left( \wt{\alpha}\right)$ is a periodic orbit freely
homotopic to $\alpha$ traversed in the backwards flow direction.

In particular, every periodic orbit of $\flot$ has infinitely many other orbits in its free homotopy class.
\end{thm}

\begin{defin}[double free homotopy class] \label{defin:double_free_homotopy}
If $\alpha$ is a periodic orbit of $\flot$, and $\wt{\alpha}$ a 
lift of $\alpha$ to $\M$, then we call the set of orbits 
$ \left\{ \pi\left( \eta^i (\wt{\alpha}) \right) \right\}_{i \in \Z}$ the 
double free homotopy class of $\alpha$.
\end{defin}

In the sequel, we will sometimes abuse notation and say that an orbit $\wt{\alpha}$ of $\hflot$ is periodic if it is a lift of a periodic orbit $\alpha$ of $\flot$. Equivalently, for an orbit $\wt{\alpha}$ of $\hflot$, being periodic means that the stabilizer of $\wt{\alpha}$ in $\pi_1(M)$ is non-trivial (and hence, infinite cyclic).

\subsection{Slitherings and the universal circle}

A {\em slithering} in $M$ is a fibration $s\colon \M \rightarrow 
{\bf S}^1$  for which $\pi_1(M)$ acts as a bundle automorphism \cite{Thurston:3MFC}.
In other words an element of $\pi_1(M)$ takes each fiber
of $s$ to a (possibly different) fiber of $s$.
This induces a foliation of $\M$, which is invariant
by $\pi_1(M)$ and hence a foliation ${\mathcal {F}}$ in $M$
which by definition is $\R$-covered. We say that this
foliation comes from a slithering or is a slithering.

Suppose that $\phi$ is a skewed, $\R$-covered Anosov flow
with $\fs$ transversely orientable.
Let $s^s \colon \M \rightarrow \leafs/ \eta^u\circ \eta^s \simeq S^1$ and $s^u \colon \M \rightarrow \leafu/ \eta^s\circ \eta^u \simeq S^1$. Note that the leaves of the foliations $\hfs$ and $\hfu$ are obtained, respectively, by taking the connected components of the fibers of $s^s$ and $s^u$. 
Then both $\fs$ and $\fu$ are slitherings.
In fact the two foliations are connected by the maps $\eta^u$ and $\eta^s$.

\vskip .1in
Thurston also introduced a {\em universal circle} for any
foliation in a $3$-manifold with hyperbolic leaves
or Gromov hyperbolic leaves. We describe the universal
circle for the foliation $\fs$ (and similarly $\fu$) as it
will be fundamental for our results.

Using Candel's result \cite{Candel:uniformization}
one may assume that the leaves of $\fs$ 
are hyperbolic. Hence, any leaf $F$ of $\hfs$ admits a circle at infinity $\partial_{\infty}F$.
Since $\fs$ is a slithering, Thurston proved that
any two leaves of $\hfs$ are a bounded distance from
each other, the bound depending on the pair of leaves
of $\hfs$
\cite{Thurston:3MFC}.
Using this and the $\R$-covered property of $\fs$ then,
for any pair $F, L$ of leaves of $\hfs$, one
produces a coarsely well defined quasi-isometry
$f \colon F \rightarrow L$. This $f$ induces a homeomorphism
between the circles at infinity of $F$ and $L$,
$f_{F,L} \colon \partial_{\infty} F \rightarrow \partial_{\infty} L$.
The map $f_{F,L}$ is characterized as follows: given $p$ an ideal
point of $F$, let $r$ be a geodesic ray in $F$ with ideal point
$p$. Then $r$ is a bounded distance from a curve $r'=f(r)$ in
$L$. The curve $r'$ is a quasigeodesic and limits in a single
point in $\partial_{\infty} L$, which is the image of $p$ by
the map $f_{F,L}$.
The map $f_{F,L}$ is canonical: it does not depend on the choice 
of map $f$.
In addition the collection of maps $f_{F,L}$ satisfies
a cocycle property, if $F, E, L$ are leaves of $\hfs$ then:

$$f_{F,L} \ \ = \ \ f_{E,L} \circ f_{F,E}$$

\noindent
The universal circle
$\univ$ of $\fs$ is the quotient of the union
of all circles at infinity of leaves of $\hfs$ by the
identifications induced  by these homeomorphisms.
The fundamental group acts naturally on the universal
circle.
Clearly there is a universal circle of the unstable
foliation as well. These are essentially the same by
a series of identifications. Therefore we will only
use the universal circle of the stable foliation
and denote it by $\univ$.

Another object which will be very useful is the 
{\em cylinder at infinity} denoted by ${\mathcal A}$.
It is the union of all circles at infinity
$\partial_{\infty} F$ of leaves of $\hfs$ topologized
so that ideal points associated to 
geodesic rays which stay near for a long
distance are near. The fundamental group naturally
acts on ${\mathcal A}$.
This cylinder obviously comes with a horizontal foliation
by circles at infinity. Using the homeomorphisms
$f_{F,L}$ described above then ${\mathcal A}$ also
has a natural vertical foliation: a vertical leaf
is the union of all points which are identified
by one of the homeomorphims $f_{F,L}$.
See details in articles by Calegari \cite{Calegari:geometry_of_R_covered} and Fenley
\cite{Fen:Foliations_TG3M}.

Thurston \cite{Thurston:3MFC} (see also Calegari \cite{Calegari:geometry_of_R_covered} and Fenley \cite{Fen:Foliations_TG3M}) proved the following:
\begin{thm}[Thurston, Fenley, Calegari]
 If $\mathcal{F}$ is a foliation coming from a slithering on an atoroidal, aspherical closed $3$-manifold $M$, then the associated universal circle $\univ$ admits two laminations $\Lambda^{\pm}_{\text{univ}}$ which are preserved under the natural action of $\pi_1(M)$ on $\univ$.
\end{thm}

From these laminations, Thurston managed to construct a $\pi_1(M)$-invariant flow on $\M$, which projects to a pseudo-Anosov flow on $M$:
\begin{thm}[Thurston \cite{Thurston:3MFC}]
 Let $\mathcal{F}$ be a slithering on an atoroidal, aspherical closed $3$-manifold $M$. 
Suppose that $\mathcal{F}$ is transversely orientable.
Then there exists a pseudo-Anosov flow $\psi^t \colon M \rightarrow M$ such that its lift $\wt{\psi}^t$ to the universal cover $\M$ is transverse to the lifted foliation $\wt{\mathcal{F}}$ and, for any $x\in \M$, the orbit $\wt{\psi}^t (x)$ intersects every leaf of $\wt{\mathcal{F}}$.
The flow $\psi^t$ is called a regulating flow for $\mathcal{F}$.
\end{thm}

This regulating pseudo-Anosov flow will be used in section \ref{sec:action_of_pi1} to study the
co-cylindrical class of periodic orbits of the Anosov flow $\flot$.

Pseudo-Anosov flows are a generalization of suspension of pseudo-Anosov diffeomorphism. Here is how we can define them (see Mosher \cite{Mosher:DS,Mosher:DS_II} for foundational work on pseudo-Anosov flows):
\begin{defin} \label{def:pseudo-Anosov}
 A flow $\psi^t$ on a closed $3$-manifold $M$ is said to be \emph{pseudo-Anosov} if it satisfies the following conditions:
\begin{itemize}
 \item For each $x\in M$, the flow line $t\mapsto \psi^t(x)$ is $C^1$, not a single point, and the tangent vector field is $C^0$;
 \item There are two (possibly) singular transverse foliations $\Lambda^s$ and $\Lambda^u$ which are two-dimensional, with leaves saturated by the flow and such that they intersect exactly along the flow lines of $\psi^t$;
 \item There is a finite number (possibly zero) of periodic orbits, called singular orbits. A leaf containing a singularity is homeomorphic to $P\times [0,1] / f$ where $P$ is a $p$-prong in the plane and $f$ is a homeomorphism from $P\times\{1\}$ to $P\times\{0\}$. In addition $p$ is at least $3$;
 \item In a stable leaf, all orbits are forward asymptotic; in an unstable leaf, they are all backward asymptotic.
\end{itemize}
\end{defin}
 
\begin{rem}
 Note that if $\flot \colon S\Sigma \rightarrow S\Sigma$ is the geodesic flow of a negatively curved surface, 
despite the fact that $S\Sigma$ is toroidal, there also exists a regulating flow for the stable and unstable foliations: just take the flow that moves the unit vectors along their fibers. However, that flow is very far from being pseudo-Anosov.
\end{rem}

In the case of skewed, $\R$-covered Anosov flows, the universal
circles of $\fs$ and $\fu$ are naturally identified with the
topological circles $C_s = \leafs/ \eta^u\circ \eta^s$ and
$C_u = \leafu/ \eta^s\circ \eta^u$ respectively.
Here is the explanation for $\fs$:
Given a point $x$ in $C_s$ lift it to a leaf $F$ of
$\hfs$ $-$ there is  a $\Z$ worth of such leaves.
There is a unique ideal point of $F$ which is the forward
ideal point of all flow lines of $\widetilde \phi_t$ in $F$.
This projects to a point in the universal circle
of $\fs$ and let this be the image of $x$ denoted by
$c(x)$. The structure of skewed, $\R$-covered
Anosov flows implies that the map $c$ is well defined
and is a homeomorphism, see details in \cite{Thurston:3MFC}
and \cite{Fen:AFM}.
Because of this we can think of the universal circle
of $\fs$ as the 
topological circle $C_s = \leafs/ \eta^u\circ \eta^s$, 
and likewise for $\fu$.
When we do not need to specify, we will write $\univ$ for the universal circle of a
slithering. Note that under the identifications above the
map $\eta^s$ induces an homeomorphism between the universal circles for $\fs$ and $\fu$.

Before getting to the core of this article here is a last remark. Even if skewed $\R$-covered Anosov flows exists in three-manifolds with vastly 
different topology, once lifted to the universal cover their picture
is very similar, as described in Figure \ref{fig:skewed_R-covered_AF}.

\begin{SCfigure}[50][h]
\centering
\scalebox{1.5}{ 
\begin{pspicture}(-2,-1)(2,7)
\psset{viewpoint=30 30 30,Decran=60}

\defFunction[algebraic]{helice}(t){cos((2*Pi/3)*t)}{sin((2*Pi/3)*t)}{t}
\psSolid[object=courbe,r=0.001,range=0 6,linecolor=blue,linewidth=0.02,resolution=360,function=helice]

\defFunction[algebraic]{geodesic}(t){(1-t)*(-0.70710) + t*cos((2*Pi/3)*3)}{(1-t)*(-0.70710) + t*sin((2*Pi/3)*3)}{3}

\defFunction[algebraic]{unstable_leaf}(s,t){(1-t)*(-sqrt(2)/2) + t*cos((2*Pi/3)*s)}{(1-t)*(-sqrt(2)/2) + t*sin((2*Pi/3)*s)}{s}
\psSolid[object=surfaceparametree,
    opacity=0.7,
    linecolor={[cmyk]{1,0,1,0.5}},
    base= 1.875 4.875 0 1,
    action = draw,
    function=unstable_leaf,
    linewidth=0.5\pslinewidth,ngrid=30 1]
\defFunction[algebraic]{stable_leaf}(s,t){(1-t)*cos((2*Pi/3)*(1+s)) + t*cos((2*Pi/3)*1)}{(1-t)*sin((2*Pi/3)*(1+s)) + t*sin((2*Pi/3)*1)}{1}
\psSolid[object=surfaceparametree,
    linecolor=red,
    base= 1.875 4.875 0 1,
    action = draw,
    function=stable_leaf,
    linewidth=0.5\pslinewidth,ngrid=20 1]
\psSolid[object=cylindrecreux,h=6,r=1,action=draw,linewidth=0.01,ngrid=3 25,opacity = 0.1,fillcolor=red,incolor=white](0,0,0)
\end{pspicture}}
\caption{We can represent $\hflot$ in the following way: $\M$ is identified with a solid cylinder where each horizontal slice is a stable leaf. On a stable leaf, the orbits of $\hflot$ are lines all pointing towards the same point on the boundary at infinity of the leaf. We represented a stable leaf, with some orbits on it, in red. The blue curve represents the point at infinity where orbits ends. It is a way of seeing $\leafs$ ``slithers''. An unstable leaf now, represented in green, is given by fixing the $(x,y)$-coordinates (i.e., the points that project to the same point on the universal circle) and taking the lines pointing towards the blue curve. Finally, if $M$ is atoroidal and aspherical, the orbits of the regulating pseudo-Anosov flow $\wt{\psi}^t$ are vertical curves inside the cylinder and stabilize the foliation by vertical straight lines on the boundary.} \label{fig:skewed_R-covered_AF}
\end{SCfigure}
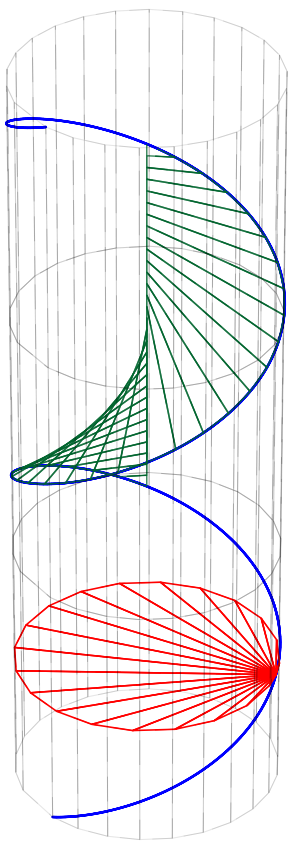

\section{Unknotting of orbits of Anosov flows in the universal cover}

Before studying isotopy classes of orbits, we warm up by showing that lifts of closed
orbits of Anosov flows on $3$-manifolds are always unknotted in the universal cover.
In fact we show unknotting behavior of any lift of an
essential curve in a leaf of a Reebless foliation in
a $3$-manifold:

\begin{prop}
Let $\mathcal F$ be a Reebless foliation in a $3$-manifold
$M$ and let $\gamma$ be an essential simple closed curve in a leaf of
$\mathcal F$. Let $\wt{\gamma}$ be any lift of $\gamma$
to $\wt{M}$. Then $\wt{\gamma}$ is unknotted
in $\wt{M}$.
Equivalently $\wt{M} - \wt{\gamma}$ 
is homeomorphic to a solid torus or
$\pi_1(\wt{M} - \wt{\gamma}) \cong \Z$.
\end{prop}

\begin{proof}
Essential means the curve is not null homotopic.
Results of Palmeira \cite{Palmeira:Open_manifolds_foliated_planes} (see also \cite{GabaiOertel}) show that, in 
the universal cover, the foliation
$\wt{\mathcal F}$ is conjugate to ${\mathcal G} \times
\R$, where $\mathcal G$ is a foliation of the plane
$\R^2$.
We do an exhaustion of $\wt{M} \cong \R^3 = 
\R^2 \times \R$ using disk times
interval in this product so that each disk intersects
a leaf of $\mathcal G$ at most once in an interval
or a point. 
We define in this way increasing balls $B_i \cong D_i \times [a_i,b_i]$ 
such that $\wt M = \bigcup_{i} B_i$, here $D_i$ is a disk.
We do this so that for $i < j$, \ $[a_i,b_i] \subset (a_j,b_j)$
and $D_i$ is a subset of the interior of $D_j$.
In addition when restricted to $B_i$ the foliation
has the form $\mathcal G_i \times [a_i,b_i]$
 where $\mathcal G_i$ is
a foliation of $D_i$, where is is assumed above that a
leaf of $\wt{\mathcal F}$ intersects $B_i$ in at most 
one component.
Fix a base point of $\wt{\gamma}$ and without
loss of generality assume it is in $B_1$ and let
$\delta_i$ be the component of $\wt{\gamma}$ in $B_i$
containing the basepoint. 
Each $\delta_i$ is contained in $\eta_i \times [a_i,b_i]$
in the product structure of $B_i$, where $\eta_i$ is
a leaf of $\mathcal G_i$. 
The curve $\delta_i$ cuts the disk 
$\eta_i \times [a_i,b_i]$ into two subdisks. 
It follows that 
$\delta_i$ is unlinked in $B_i$, or equivalently
that $\pi_1(B_i - \delta_i) \cong \Z$, or
$B_i - \delta_i$ is a solid torus.

This is true for every $i$. Using the exhaustion
$B_i, i \in \N$ and the inclusion maps
$\delta_i \rightarrow \delta_j$, $B_i \rightarrow
B_j$ if $i < j$, the result follows from 
standard elementary considerations.
\end{proof}

Notice that if $\gamma$ is a closed orbit of an Anosov flow, then $\fs$ is a Reebless foliation and
$\gamma$ is essential in its stable leaf, so we have
\begin{cor}
 Let $\gamma$ be a closed orbit of an Anosov flow on a $3$-manifold. Any lift $\wt{\gamma}$ of $\gamma$ is unknotted in $\wt{M}$.
\end{cor}

\begin{rem}
The same results holds if $\gamma$ is a closed orbit
of a pseudo-Anosov flow. The stable foliation blows up
to an essential lamination $\mathcal L$ in $M$.
Gabai and Oertel (\cite{GabaiOertel}) proved
that if $\mathcal L$ is an essential lamination in
a three manifold $M$, then in the universal cover 
$\wt{\mathcal L}$ has the form 
$\mathcal G \times \R$, where $\mathcal G$ is
an essential lamination of the plane, i.e., 
without closed leaves or monogons. The same proof as in the proposition 
applies to yield the result.
\end{rem}

\section{Isotopy classes of periodic orbits: Proof of Theorem A}

This section is devoted to the proof that homotopic orbits are isotopic. Let us state precisely what we mean by isotopy.
\begin{defin}
 Two curves $c_1$ and $c_2$ in $M$ are isotopic if there exists a continuous map $H \colon S^1 \times [0,1] \rightarrow M$ such that $H(S^1,0)= c_1$, $H(S^1,0)= c_2$ and, for any $t\in [0,1]$, $H(S^1,t)$ is an embedding of $S^1$ in $M$.
\end{defin}

We stress that it is not required that
$H$ is an embedding. Equivalently it is not required that
$H({\bf S}^1 \times [0,1])$ is an 
\emph{embedded} annulus connecting $c_1$ and
$c_2$.

Before starting the proof of Theorem A, let us just reformulate it a bit more precisely.
\begin{thm}\label{thm:homo_implies_isotope} 
 Let $\flot$ be a skewed $\R$-covered Anosov flow on a closed $3$-manifold,
so that $\fs$ is transversely orientable.
If $\{ \alpha_i \}_{i \in \Z}$ is a double free homotopy class of 
periodic orbits of $\flot$ (see Definition \ref{defin:double_free_homotopy}), then all the $\alpha_i$s are isotopic.
\label{isotopic}
\end{thm}

\begin{proof}[Proof of Theorem \ref{thm:homo_implies_isotope}]
 We are going to construct an isotopy between $\alpha_0$ and $\alpha_1$. As isotopy is an equivalence relation, it will show that all free homotopic orbits are isotopic.

First we construct a particular type of immersed annulus from $\alpha_0$
to $\alpha_1$.
Let $\al 0$ a lift of $\alpha_0$ to $\M$ and ${\al 1 = \eta(\al 0)}$. 
Let $F_i = \hfs(\al i), i = 0, 1$.
Let $g$ be the element of $\pi_1(M)$ associated to $\alpha_0$
leaving $\al 0$ invariant and with attracting fixed point
the forward ideal point of 
$\wt{\alpha_0}$ (considering the action of $g$
on $F_0 \cup \partial_{\infty} F_0$).
The set of leaves of $\hfs$ between $F_0$ and $F_1$ 
forms a closed interval $I$ in the stable leaf space $\leafs$
which we parametrize as
$$I \ = \ \{ F_t \ | \ 0 \leq t \leq 1 \}$$

\noindent
The closed orbit $\alpha_0$ is isotopic to a geodesic
$\tau_0$ in $\pi(F_0)$ $-$ this uses that $\fs$ is 
transversely orientable, which in our situation implies that
the non planar leaves of $\fs$ are annuli.
Let $c_0$ be the lift of $\tau_0$ to $F_0$ with
same ideal points $a, b$ as $\al 0$.
In $F_t$, consider the geodesic $c_t$ with ideal points
which are associated to $a, b$ using the universal
circle identification. The $c_t$ vary continuously with
$t$. 
Let 
$$\widetilde C \ := \bigcup \ \{ c_t \ | \ t \in [0,1] \}$$

\noindent
This set is invariant by $g$ and hence $\wt{C}/g$ is a compact
annulus. Notice that $\tau_1 := \pi(c_1)$ is a closed curve
in $\pi(F_1)$ which is isotopic to $\alpha_1$. 
In addition $\tau_0, \tau_1$ are simple closed
curves (geodesics) in their respective stable
leaves. 

We will show that $\tau_0 = \pi(c_0)$ and $\tau_1 = \pi(c_1)$ are isotopic.
To do that we will use  the set
$C := \pi(\widetilde C)$ which  is a compact annulus in
$M$, but a priori only immersed.
If $C$ is embedded this is obvious. In the general case we will
carefully analyze the self intersections of $C$ to show
that it still produces a desired isotopy.



\begin{claim}
There are only a finite number of self-intersections of $C$ on $\tau_0$,
except perhaps for the two boundary components of $C$
being identified.
\end{claim}

\begin{proof}
It suffices to show that $C$ is transverse to itself.
A self intersection corresponds to an element $h$ of $\pi_1(M)$
so that $h(\wt{C})$ intersects $\wt{C}$.
So there is some $c_t$ with $h(c_t)$ intersecting
some other $c_{t'}$. Notice that all $c_t$ are geodesics
in their respective leaves, and hence $h(c_t)$ is also
a geodesic in $F_{t'}$. 
Suppose by way of contradiction that the intersection
of $h(c_t)$ and $c_{t'}$ is not transverse.
As both of these are geodesic in a hyperbolic plane,
then $h(c_t) = c_{t'}$. 
But then $h$ leaves invariant both points of the universal circle
$\univ$ associated to $a, b$. 

This also implies that either $h$ or $h^{-1}$ 
takes $c_0$ or $c_1$ inside $\wt{C}$.
For simplicity suppose that $h(c_0)$ is contained
in $\wt{C}$.
Then $h(c_0) = c_{t"}$ for some $t"$ in $(0,1]$. 
Hence $c_{t"}$ also projects in $M$ to $\tau_0$ and
$ \bigcup \{ c_t \ | \ 0 \leq t \leq t" \}$ projects
to an a priori immersed torus in $M$. 
This torus is a free homotopy from $\tau_0$ to itself.
So
$\bigcup \{ c_t \ | \ 0 \leq t \leq t" \}$ is
the lift of this free homotopy and hence $F_{t"}$ 
is invariant under $g$. But the only leaves
invariant under $g$ in the interval $I$ of
$\hfs$ are $F_0$ and $F_1$. It follows that
$t" = 1$. 
This corresponds to the two boundary components being identified.
This proves that $C$ is transverse to itself except
perhaps at the boundary and finishes the proof of the claim.
\end{proof}





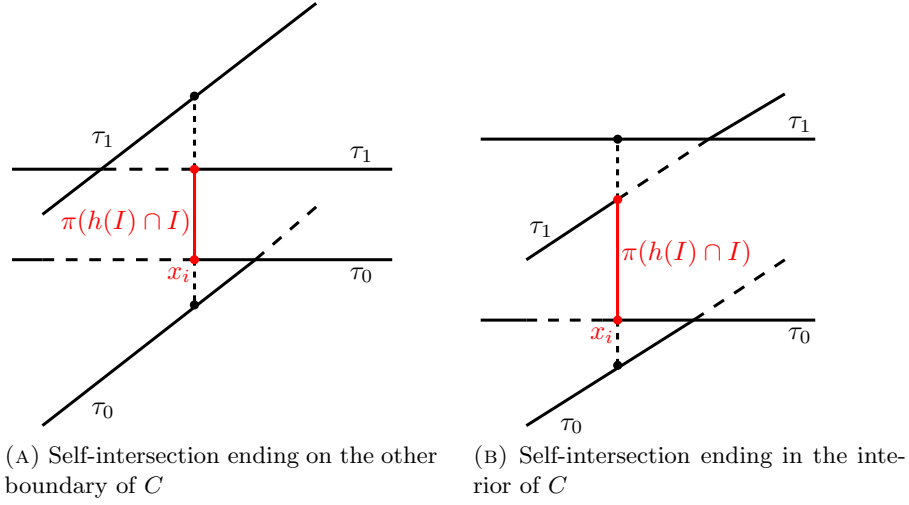
\begin{figure}[h]
\begin{subfigure}[b]{0.45\textwidth}
\centering
\begin{pspicture}(0,-2.82)(5.5,2.82)
\psline[linewidth=0.04cm](0.4,0.0)(4.0,2.8)
\psline[linewidth=0.04cm](0.4,-2.8)(3.2,-0.6)
\psline[linewidth=0.04cm](2.4,0.6)(5.0,0.6)
\psline[linewidth=0.04cm](2.4,-0.6)(5.0,-0.6)
\psline[linewidth=0.04cm,linestyle=dashed,dash=0.16cm 0.16cm](3.2,-0.6)(4.0,0.1)
\psline[linewidth=0.04cm,linestyle=dashed,dash=0.16cm 0.16cm](2.4,-0.6)(0.4,-0.6)
\psline[linewidth=0.04cm,linestyle=dashed,dash=0.16cm 0.16cm](2.4,0.6)(1.2,0.6)
\psline[linewidth=0.04cm](1.2,0.6)(0.0,0.6)
\psline[linewidth=0.04cm](0.4,-0.6)(0.0,-0.6)
\psline[linewidth=0.04cm,linecolor=red](2.4,0.6)(2.4,-0.6)
\psline[linewidth=0.04cm,linestyle=dashed,dash=0.08cm](2.4,-0.6)(2.4,-1.23)
\psline[linewidth=0.04cm,linestyle=dashed,dash=0.08cm](2.4,0.6)(2.4,1.54)
\rput(4.6,-0.8){$\tau_0$}
\rput(4.6,0.8){$\tau_1$}
\rput(1.2,-2.6){$\tau_0$}
\rput(1.2,1){$\tau_1$}
\rput(1.5,-0.1){\color{red} $\pi(h(I) \cap I)$}
\psdots[dotsize=0.12,linecolor=red](2.4,-0.6)
\psdots[dotsize=0.12,linecolor=red](2.4,0.6)
\rput(2.2,-0.8){\color{red} $x_i$}
\psdots[dotsize=0.12](2.4,-1.2)
\psdots[dotsize=0.12](2.4,1.57)
\end{pspicture} 
  \caption{Self-intersection ending on the other boundary of $C$}
\label{fig:self_inter_a} 
\end{subfigure} \quad
\begin{subfigure}[b]{0.45\textwidth}
 \centering
\begin{pspicture}(0,-2.22)(5.5,2.22)
\psline[linewidth=0.04cm](1.6,-0.8)(4.4,-0.8)
\psline[linewidth=0.04cm](0.6,-2.2)(2.8,-0.8)
\psline[linewidth=0.04cm,linestyle=dashed,dash=0.16cm 0.16cm](2.8,-0.8)(4.0,0.0)
\psline[linewidth=0.04cm,linestyle=dashed,dash=0.16cm 0.16cm](1.6,-0.8)(0.6,-0.8)
\psline[linewidth=0.04cm](0.6,0.0)(1.8,0.8)
\psline[linewidth=0.04cm,linestyle=dashed,dash=0.16cm 0.16cm](1.8,0.8)(3.0,1.6)
\psline[linewidth=0.04cm](0.0,1.6)(4.4,1.6)
\psline[linewidth=0.04cm](0.0,-0.8)(0.6,-0.8)
\psline[linewidth=0.04cm](3.0,1.6)(4.0,2.2)
\psline[linewidth=0.04cm,linestyle=dashed,dash=0.08cm](1.8,-1.4)(1.8,-0.8)
\psline[linewidth=0.04cm,linestyle=dashed,dash=0.08cm](1.8,0.8)(1.8,1.6)
\psline[linewidth=0.04cm,linecolor=red](1.8,-0.8)(1.8,0.8)
\rput(4.2,-1){$\tau_0$}
\rput(4.2,1.8){$\tau_1$}
\rput(1.2,-2.2){$\tau_0$}
\rput(0.8,0.4){$\tau_1$}
\put(1.85,0){\color{red} $\pi(h(I) \cap I)$}
\psdots[dotsize=0.12,linecolor=red](1.8,-0.8)
\psdots[dotsize=0.12,linecolor=red](1.8,0.8)
\rput(1.6,-1){\color{red} $x_i$}
\psdots[dotsize=0.12](1.8,-1.4)
\psdots[dotsize=0.12](1.8,1.6)
\end{pspicture} 
  \caption{Self-intersection ending in the interior of $C$}
 \label{fig:self_inter_b}
\end{subfigure}
 \caption{The possible self-intersections of $C$} \label{fig:self_inter}
\end{figure}

\begin{rem}
The possible tangent self intersection at the boundary of
$C$ will not be a problem for the proof of the theorem.
\end{rem}

In our analysis it is useful to
differentiate between an actual annulus $\widetilde C/g$
and the immersed annulus $C$.
 Let $x_1, \dots, x_n$ be all the self-intersections of $C$ on 
$\tau_0$ except for possible boundary tangency. 
Now we analyze in detail the self intersections of $C$, except
for the possible boundary tangency.
As described in the previous claim any such self intersection
lifts to an intersection between $\widetilde C$
and $h(\widetilde C)$ for some $h$ in $\pi_1(M)$.
Hence $h(I) \cap I \not = \emptyset$.
Using $h^{-1}$ instead of $h$ if necessary we may assume
that $h(F_0)$ is in the interval $I$. 
As $\tau_0$ is embedded then $h(F_0)
= F_d$ for some $d$ in $(0,1]$. 
In addition the claim proves that 
$h(c_0)$ intersects $c_d$ transversely.
Let $a_0, b_0$ be 
the points in the universal circle $\univ$ associated
to the ideal point $a, b$ of the geodesic $c_0$ of $F_0$.
Let $a_1, b_1$ be the points in $\univ$ associated
to the ideal points of $h(c_0)$ in the leaf $F_d$.
The fact that $h(c_0)$ and $c_d$ intersect transversely
is equivalent to the pairs \  $a_0, b_0$ \ and \ $a_1, b_1$ \ being
linked in the universal circle $\univ$ 
(i.e., $a_1$ and $b_1$ are in the two distinct connected components of $\univ \smallsetminus \{a_0,b_0\}$).
This shows that the self intersections of $C$ lifted to $\wt{C}/g$ are
compact arcs in $\wt{C}/g$ starting in one boundary of $\wt{C}/g$ and either
\begin{itemize}
 \item ending in the interior of $C$ 
 \item or ending in the other boundary component of $\wt{C}/g$.
\end{itemize}

These two possibilities correspond in the
description above to 
either $h(I) \not \subset I$, see Figure \ref{fig:self_inter_b} 
or $h(I) \subset I$ see Figure \ref{fig:self_inter_a}.

\begin{claim} \label{claim_hI_not_in_I}
We can never have $h(I) \subset I$.
\end{claim}

\begin{proof}
In order to analyze this we will explore in more detail the
structure of skewed $\R$-covered Anosov flows as related
to our situation.

The flow line $\wt{\alpha_0}$ has the same
ideal points as the geodesic $c_0$ in $F_0$.
Recall that these ideal points in
$\partial_{\infty} F_0$ are $a, b$ and assume that
$a$ is the forward ideal point of $\wt{\alpha_0}$.

Let ${\mathcal C}$ be the lozenge with corners 
$\wt{\alpha_0}, \wt{\alpha_1}$.
The lozenge ${\mathcal C}$ has stable sides in
half leaves of $F_0, F_1$ and unstable sides
in $U_0 = \hfu(\wt{\alpha_0})$ and
$U_1 = \hfu(\wt{\alpha_1})$.
We put an orientation in the 
ideal boundary
$\partial_{\infty} F_0$, so that the
interval $(b,a)$  (that is from $b$ to $a$ in
the positive orientation) corresponds to the negative ideal
points of flow lines in the stable half
leaf of $(F_0 - \wt{\alpha_0})$ which is contained
in the boundary of the lozenge ${\mathcal C}$.
The orientation in $\partial_{\infty} F_0$ induces 
an orientation in $\univ$.
Denote the interval $(b_0,a_0)$ in $\univ$
by $J_1$ and let $J_2 = (a_0,b_0)$.
The ideal points in $J_1$  correspond to the distinguishing
negative ideal points of the unstable leaves
intersecting ${\mathcal C}$.

\begin{figure}[h]
\begin{subfigure}[b]{0.45\textwidth}
\scalebox{1}{
 \begin{pspicture}(0,-2.3)(4.6,2.9)
\pscircle[linewidth=0.04,dimen=outer](2.3,0.0){2.3}
\psbezier[linewidth=0.04,ArrowInside=->,ArrowInsidePos=0.7,arrowsize=0.3](4.56,0.28)(3.3,0.38)(1.04,0.08)(0.04,-0.3)
\psbezier[linewidth=0.04,ArrowInside=->,ArrowInsidePos=0.5,arrowsize=0.3](2.66,2.22)(2.26,1.08)(1.7,0.38)(0.06,-0.28)
\psbezier[linewidth=0.04,linestyle=dashed,dash=0.16cm 0.16cm,ArrowInside=->,ArrowInsidePos=0.8,arrowsize=0.3](2.62,2.22)(2.08,1.46)(1.16,0.86)(0.24,0.96)
\psbezier[linewidth=0.04,linestyle=dashed,dash=0.16cm 0.16cm,ArrowInside=->,ArrowInsidePos=0.7,arrowsize=0.3](2.66,2.22)(2.5,1.06)(2.68,-1.12)(3.18,-2.1)

\rput(3.4,0){$\widetilde{\alpha}_0$}
\put(3,-1.2){$\gamma$?}
\rput(1.4,1.4){$\gamma$?}
\rput(1,0.5){$\gamma_0$}

\psdots[dotsize=0.14](4.56,0.28)
\psdots[dotsize=0.14](0.04,-0.3)
\psdots[dotsize=0.14](2.66,2.24)

\uput{4pt}[0](4.56,0.28){$b_0$}
\uput{4pt}[80](2.66,2.24){$b_2$}
\uput{4pt}[190](0.04,-0.3){$a_0$}

\rput(4.1,1.8){$J_1$}
\rput(4.15,-1.75){$J_2$}
\rput(0.6,-1.9){$\univ$}
\end{pspicture}  }
\caption{The possible ending points of $\gamma$}
\label{fig_is1a}
\end{subfigure}
\begin{subfigure}[b]{0.45\textwidth}
\scalebox{1}{
 \begin{pspicture}(0,-2.62)(5.42,2.62)
\psline[linewidth=0.04cm,linestyle=dashed](0.0,-1.0)(3.6,2.6)
\psline[linewidth=0.04cm,linestyle=dashed](1.6,-2.6)(5.4,1.2)
\psline[linewidth=0.04cm](1.8,-2.4)(1.8,0.8)
\psline[linewidth=0.04cm](0.2,-0.8)(3.4,-0.8)
\psline[linewidth=0.04cm](3.4,-0.8)(3.4,2.4)
\psline[linewidth=0.04cm](1.8,0.8)(5.0,0.8)
\psline[linewidth=0.04cm](1.0,0.0)(4.2,0.0)
\psdots[dotsize=0.14](1.8,0.0)
\psdots[dotsize=0.14](2.6,0.0)
\psdots[dotsize=0.14](1.8,-0.8)
\psdots[dotsize=0.14](3.4,0.8)

\uput{4pt}[80](2.6,0.0){$\gamma$}
\uput{4pt}[-135](1.8,0.0){$\gamma_0$}
\uput{4pt}[-135](1.8,-0.8){$\widetilde{\alpha}_0$}
\uput{4pt}[45](3.4,0.8){$\widetilde{\alpha}_1$}

\put(3,-0.5){\Large $\bf \mathcal{C}$}
\rput(3.8,0.2){$U$}
\put(2.4,-1.1){$U_0$}
\put(2.6,0.9){$U_1$}
\rput(3.7,2){$F_1$}
\put(1.4,-2.2){$F_0$}
\end{pspicture} }
\caption{An orbit $\gamma$ in $\mathcal C$}
\label{fig_is1b}
\end{subfigure}
\caption{}
\label{fig_is1}
\end{figure}

We want to determine exactly what are the orbits
in the lozenge ${\mathcal C}$, i.e., given an orbit $\gamma \in \mathcal C$ we want to understand where the forward ideal point of $\gamma$ is, relatively to the forward ideal point of $\wt{\alpha}_0$ (see Figure \ref{fig_is1}).
Figure \ref{fig_is1a} is a schematic, but not rigorous,
drawing of several orbits
in question. The boundary circle is the universal circle
and the orbits in question represent several orbits in
$\wt{M}$ whose endpoints in their respective
leaves corresponds to those points in $\univ$
which are the ideal points of the curve in the diagram.
The diagram is not precise for two reasons. First, not all orbits are
in the same leaf of the stable foliation, so this
is not in a single leaf but in a superposition of leaves each drawn on tracing paper. And second, two ideal points
of $\univ$ do not define a single flow line in
$\widetilde M$ (indeed, two points in $\univ$ defines a $\Z$ worth of flow lines).
Still this schematic drawing will be extremely useful.
In the case of the geodesic flow this drawing
can be made precise as the universal cover of the
surface with $\univ$ the ideal circle of this
hyperbolic plane.

Let $\gamma$ be
an orbit in ${\mathcal C}$. 
Then $\hfu(\gamma)$ intersects
the half leaf of $F_0 - \wt{\alpha_0}$ in an orbit that we call $\gamma_0$.
The negative ideal point of
$\gamma_0$ is in $(b,a)$ and hence it corresponds
to a point in $J_1$, by definition.

\vskip .1in
Recall that the identification of $\univ$ with
the circle
$C_s = \leafs/\eta^u \circ \eta^s$ means that we are at least
locally parametrizing the stable leaf by the ideal
point in its boundary (or point in $\univ$) which is the forward
ideal point of all flow lines in that stable leaf.
The orbit $\gamma$ is not in $F_0$ so we cannot strictly 
draw it in the same leaf $F_0$. Intuitively we think of
it as connecting two ideal points in its stable leaf
and hence connecting two points of the universal circle
$\univ$. In this extended meaning we can still
draw the orbit $\gamma$ in the same diagram which has
$\wt{\alpha_0}$ and $\gamma_0$.
Let the points in $\univ$ defined by $\gamma$ be
$a_2, b_2$, with $a_2$ the positive ideal point (see Figure \ref{fig_is1a}). 
The key is to determine
whether $a_2$
is in $J_1$ or $J_2$.

The ideal points $a_0, b_2$ of
$\gamma_0$ (as a point in $\univ$) define
two complementary intervals in $\univ$. These intervals are $N_1:= (b_2,a_0)$, and $N_2 := (a_0,b_2)$. The interval
$N_2$ contains $J_2$ and $N_1$ is contained in 
$J_1$ --- this is because $b_2$ is in $(b_0,a_0)$.
The interval $N_1$ consist of the forward ideal points of the flow lines in the half-leaf of $\hfu(\gamma_0) $ on one side of $\gamma_0$ and $N_2$ corresponds to those on the other side. Let $U$ be the
half leaf of $\hfu(\gamma_0)$ that intersects the lozenge $\mathcal C$. Suppose that $\gamma_0$ ends in $N_1$, then $U$ is contained in $N_1$, and hence, contained in $J_1$.
But $J_1$ consist of all the forward ideal points of flow lines in the half-leaf of $U_0$ contained in the boundary of $\mathcal C$. Which forces $U$ to be contained in the lozenge $\mathcal C$, which is obviously
not true as seen in Figure \ref{fig_is1b}.

We deduce that $\gamma$ has forward ideal point
in the interval $N_2$. We still need to show that the forward ideal
point is in $J_2$ and not in $J_1$.
The orbit $\gamma$ is in the lozenge $\mathcal C$ and hence
$\hfs(\gamma)$ intersects $\hfu(\wt{\alpha_0})$.
The positive ideal points of flow lines in the side of
${\mathcal C}$ contained in $U_0$ are exactly those in
$J_2$. This shows that the positive ideal point of
$\gamma$ is in $J_2$.
In fact $J_2$ is exactly the interval of positive ideal
points of such $\gamma$ in ${\mathcal C}$.
Therefore:

\vskip .1in
\noindent
{\bf Conclusion:} \ Locally the orbits in ${\mathcal C}$ 
are exactly those that have negative
ideal point in $J_1$ and forward ideal point in $J_2$.
\vskip .1in

Now we return to the proof of the claim. 
Let $h$ be in $\pi_1(M)$ with $h(c_0)$ 
intersecting $c_d$ transversely, so that $h(I) \cap I \not = \emptyset$.
In other words $h$ sends the ideal points $a_0, b_0$
of $\wt{\alpha_0}$ (seen as points in $\univ$) to
points in $J_1, J_2$ respectively or vice versa.

We flesh this out a little more. 
Once again parametrizing the stable leaf in $I$ with their ideal points, 
we see that the interior of $I$ projects to the interval $J_2 = (a_0,b_0)$ in $\univ$. 
Since $h(F_0) = F_d$ for
some $d > 0$ then the positive ideal point of $h(\wt{\alpha_0})$ has
to be in $J_2$. It follows that the negative ideal
point of $h(\wt{\alpha_0})$ is in $J_1$ and by the
conclusion above $h(\wt{\alpha_0})$ is an orbit
in the lozenge ${\mathcal C}$. 

The interior of $I$ corresponds to the half leaf $Z$.
of $\hfu(\wt{\alpha_0})$ in the boundary of the
lozenge $\mathcal C$. Hence $h(I) \subset I$ implies
that $h(Z)$ is contained in the lozenge $\mathcal C$.
But no half leaf is contained in a lozenge,
so this is a contradiction.
This finishes the proof of Claim \ref{claim_hI_not_in_I}.
\end{proof}

Let us recap what we have proved so far: the transverse
self intersections
of $C$ when lifted to $\wt{C}/g$ are a finite number of compact arcs $\delta_i$
so that each arc starts in one of $c_0/g$ or $c_1/g$ and
ends in the interior of $\wt{C}/g$.
We will trivialize the annulus 
$\wt{C}/g$ in order to produce an isotopy
from $\tau_0$ to $\tau_1$.

First produce a vertical trivialization of $\wt{C}/g$ $-$ that is,
a vertical foliation of the annulus
$\wt{C}/g$ by compact arcs $-$
so that \ I) every arc $\delta_i$ is in a vertical fiber
and II) distinct arcs are in different vertical fibers.
Now we produce a horizontal trivialization of 
$\wt{C}/g$ $-$ a horizontal foliation by simple
closed curves as follows.
The horizontal trivialization is chosen to be
transverse to the vertical foliation and to satisfy
the following property.
Suppose that $\delta_i$ is a curve of self intersection of
$C$, which, when lifted to $\wt{C}/g$, is contained in
$x \times [0,1]$ and also in $y \times [0,1]$ in the
double trivialization of $\wt{C}/g$. 
We assume that it is contained in $x \times [0,e_0]$ and
also in $y \times [e_1,1]$. We do the trivialization so that
when mapped into $M$ the point $(x,t)$ is never mapped
to the same point as $(y,t')$ where $t' < t$. In other
words when pulled back to
$\wt{C}/g$, the pre image starts first
in $x \times [0,1]$ and then in $y \times [0,1]$. The condition
above guarantees that as $t$ changes the points in
$y \times [0,1]$ never catch up with the points in 
$x \times [0,1]$ when mapped to
$C$ in $M$.
See Figure \ref{fig_horizontal_trivia}, 
where we assumed that $\delta_i$ was the only self-intersection in the range of the drawing.
The $\delta_i$ in the figure represent the two pullbacks
of $\delta_i$ to $\wt{C}/g$.

\begin{figure}[h]
 \begin{pspicture}(0,-1.8)(6,1.8)
\psline[linewidth=0.04cm](0.0,-1.5)(6.0,-1.5)
\psline[linewidth=0.04cm](0.0,1.5)(6.0,1.5)
\psline[linewidth=0.04cm,linecolor=red](1.2,-1.5)(1.2,0.5)
\psline[linewidth=0.04cm,linestyle=dashed,dash=0.16cm 0.16cm](1.2,0.5)(1.2,1.5)

\psline[linewidth=0.04cm,linestyle=dashed,dash=0.16cm 0.16cm](4.0,-1.5)(4.0,-0.3)
\psline[linewidth=0.04cm,linecolor=red](4.0,-0.3)(4.0,1.5)
%
%

\psline[linewidth=0.04cm](0.0,-1)(6.0,-1)
\psline[linewidth=0.04cm](0.0,-0.5)(6.0,-0.5)
\psline[linewidth=0.04cm](0.0,0)(6.0,0)
\psline[linewidth=0.04cm](0.0,0.5)(6.0,0.5)
\psline[linewidth=0.04cm](0.0,1)(6.0,1)

\psdots[dotsize=0.12,linecolor=red](1.2,-1.5)
\rput(0.8,-1.75){$(x,0)$}
\psdots[dotsize=0.12,linecolor=red](1.2,0.5)
\rput(0.7,0.7){$(x,e_0)$}
\psdots[dotsize=0.12](1.2,1.5)
\rput(0.8,1.75){$(x,1)$}

\rput(1,-0.3){$\color{red} \delta_i$}

\psdots[dotsize=0.12,linecolor=red](4.0,-0.3)
\put(4.1,-0.35){$(y,e_1)$}
\psdots[dotsize=0.12,linecolor=red](4.0,1.5)
\put(4.1,1.65){$(y,1)$}
\psdots[dotsize=0.12](4.0,-1.5)
\put(4.1,-1.8){$(y,0)$}

\put(4.1,0.7){$\color{red} \delta_i$}
\end{pspicture}  
\caption{The horizontal trivialization in $\wt{C}/g$}
\label{fig_horizontal_trivia}
\end{figure}
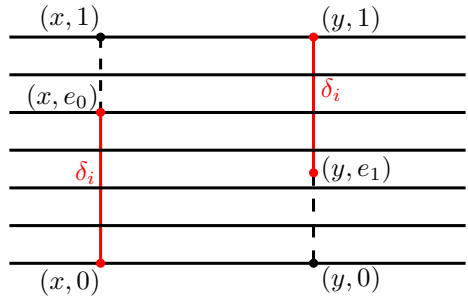

Now in $\wt{C}/g$ consider the isotopy
from $c_0/g$ to $c_1/g$ given 
by
$$H_t(z,0) \ = \ (z,t)$$

\noindent
Project this homotopy to $M$. The condition above implies that
the images $H_t({\bf S}^1 \times t)$ never self intersect. The only places where they could self intersect would
be at the self intersections of $C$. At these arcs of self
intersection, the condition above rules out the self intersection.

We conclude that the projection to $M$ of the map $H$ is an
isotopy from $\tau_0$ to $\tau_1$. 
Notice that the possible tangent self intersections
of $C$ are not a problem for this argument as they
happen when $t = 0$ and $t = 1$.
This finishes the proof of theorem \ref{isotopic}.
\end{proof}

\begin{rem}
We now comment on the hypothesis that $\fs$ is
transversely orientable. Suppose that $\fs$ is not
transversely orientable and let $g$ in $\pi_1(M)$ 
reversing local orientation. Then $g$ reverses
orientation in $\leafs$ and so has a unique
fixed leaf, call it $F_0$. Then $g$ leaves invariant
a flowline $\gamma_0$ in $F_0$. The square
$g^2$ fixes a ${\bf Z}$ worth of leaves
$F_i$ in $\leafs$. For each $i$ there is a flowline
$\gamma_i$ in $F_i$ so that $g^2(\gamma_i) =
\gamma_i$. Let $\alpha_i = \pi(\gamma_i)$.
The problem is that $\alpha^2_0$ is freely homotopic
to $\alpha^2_i$ for any $i$, but 
$\alpha_0$ is {\underline {not}} freely homotopic
to $\alpha_i$. 
One can conceivably consider $\alpha_0^2$ instead
of $\alpha_0$. As the stable leaf of $\alpha_0$ is
a Mobius band then $\alpha^2_0$ is homotopic to
a simple closed curve. However it is not a priori
true that the stable leaf of $\alpha_1$ also is
a Mobius band, so $\alpha^2_1$ may not be represented
by a simple curve in its stable leaf. 
\end{rem}

\section{Co-cylindrical class} \label{sec_co-cylindrical_class}

Among isotopic orbits of a pseudo-Anosov flow, we define:
\begin{defin}
 Two curves $c_1$ and $c_2$ in $M$ are \emph{co-cylindrical} if there exists an embedded annulus $A$ in $M$ such that $\partial A = c_1 \cup c_2$.
\end{defin}
Notice that this is \emph{not} an equivalence relation as it is clearly non-transitive. However, as we will see, its study is quite interesting.
We will restrict to 
skewed, $\R$-covered Anosov flows with
$\fs$ transversely orientable.

The existence of co-cylindrical orbits will be linked to the action of the fundamental group on the orbit space, or more precisely, the action of the fundamental group on the chain of lozenges defined by two orbits.
\begin{defin}
 Let $C = \bigcup L_i$ be a chain of lozenges. The chain $C$ is said to be \emph{simple} if the orbit of the corners under $\pi_1(M)$ does not intersect $C$.
\end{defin}

We will now show the link between having two co-cylindrical periodic orbits and simple chain of lozenges. This is essentially based on Barbot's work \cite{Bar:MPOT}.


In \cite{Bar:MPOT} (see also \cite{BarbotFenley}) Barbot studied embedded tori in (toroidal) $3$-manifolds supporting skewed $\R$-covered Anosov flows, showing that they could be put in a quasi-transverse position (i.e., transverse to the flow, apart from along some periodic orbits). We will use his work to obtain properties of embedded annuli:
\begin{thm} \label{thm:simple_and_cocylindrical}
 Let $\alpha$ and $\beta$ be two orbits in the same free homotopy class, choose coherent lifts $\wt{\alpha}$ and $\wt{\beta}$, and denote by $B(\wt{\alpha}, \wt{\beta})$ the chain of lozenges between $\wt{\alpha}$ and $\wt{\beta}$.\\
 If $\alpha$ and $\beta$ are co-cylindrical, then $B(\wt{\alpha}, \wt{\beta})$ is simple, i.e., if we denote by $(\al i)_{i=0 \dots n}$ the corners of the lozenges in $B(\wt{\alpha}, \wt{\beta})$, with $\al 0 = \wt{\alpha}$ and $\al n= \wt{\beta}$, then
\begin{equation*}
 \left( \pi_1(M) \cdot \al i \right) \cap B(\wt{\alpha}, \wt{\beta}) = \emptyset
\end{equation*}(using once again the two identifications).
Conversely, if $B(\wt{\alpha}, \wt{\beta})$ is simple, then there exists an embedded annulus, called a Birkhoff annulus, with boundary $\alpha \cup \beta$.
\end{thm}

\begin{proof}
 Construction of an embedded Birkhoff annulus from a simple chain of lozenges is done in \cite{Bar:MPOT}, hence proving the converse part.

To prove that, if $\alpha$ and $\beta$ are co-cylindrical, then $B(\wt{\alpha}, \wt{\beta})$ is simple, we have to re-prove Lemma 7.6 of \cite{Bar:MPOT} (or equivalently step 1 of the proof of Theorem 6.10 of \cite{BarbotFenley}) when, instead of having an embedded torus, we just have an embedded cylinder.\\
 Let $C$ be an embedded cylinder such that $\partial C = \lbrace \alpha, \beta \rbrace$ and $\wt{C}$ the lift of $C$ in $\M$ such that its boundary is on $\wt{\alpha}$ and $\wt{\beta}$. Let us also denote the generator of the stabilizer of $\wt{\alpha}$ by $\gamma \in \pi_1(M)$. Following \cite{Bar:MPOT}, we can construct a embedded plane $\wt{C}_0$ in $\M$ such that
\begin{itemize}
 \item $\wt{C}_0$ is $\gamma$-invariant, 
 \item $\wt{C}_0$ contains all the $\al i$,
 \item $\wt{C}_0$ is transverse to $\hflot$ except along the $\al i$,
 \item the projection of $\wt{C}_0$ to $\orb$ is $B(\wt{\alpha}, \wt{\beta})$.
\end{itemize}
Barbot's trick to obtain such a plane is, for every lozenge in $B(\wt{\alpha}, \wt{\beta})$, to take a simple curve $\bar{c}$ from one corner of the lozenge to the other (for instance $\al i$ and $\al{i+1}$). Then, lift $\bar{c}$ to $\wt{c}\subset \M$ such that $\wt{c}$ is transverse to $\hflot$. Now, choose an embedded rectangle $R_i$ in $\M$ such that $R_i$ is bounded by $\wt{c}$, $\gamma \cdot \wt{c}$, and the two pieces of $\al i$ and $\al{i+1}$ between the endpoints of $\wt{c}$ and $\gamma \cdot \wt{c}$. Then define $\wt{C}_0$ as the orbit under $\gamma$ of the unions of the rectangles $R_i$.

Replacing $\M$ by the subset of $\M$ delimited by $\hfs(\wt{\alpha})$ and $\hfs(\wt{\beta})$ and containing $\wt{C}_0$, we can copy verbatim the proof of \cite[Theorem 6.10, step 1]{BarbotFenley} and obtain that $B(\wt{\alpha}, \wt{\beta})$ is simple.
\end{proof}

Using the theorem, we can deduce the following property of co-cylindrical classes:
\begin{prop}\label{prop:cardinality_isotopy_class}
If the co-cylindrical class of one orbit is finite, then all the co-cylindrical classes in the same double free homotopy class are finite. Moreover, they all have the same cardinality.
\end{prop}

\begin{proof}
This result just relies on the fact that the image of a chain of lozenges under a deck transformation is a chain of lozenges. Equivalently, we use that the homeomorphism $\eta$ of $\orb$, defined in Proposition \ref{prop:eta_s_eta_u}, commutes with the action of $\pi_1(M)$.

 Let $\left\{\alpha_i\right\}_{i\in \Z}$ be a double free homotopy class of periodic orbit, and take $\left\{\al i\right\}_{i\in \Z}$  a coherent lift to $\M$. Suppose that $\alpha_0$ is such that $B(\al 0, \al{k-1})$ is simple and $B(\al 0, \al k)$ is non-simple. That is, there exists an element $h\in \pi_1(M)$ such that $h \cdot \al 0$ is in $L(\al{k-1}, \al k)$, the lozenge with corners $\al{k-1}$ and $\al k$. Thanks to Theorem \ref{thm:simple_and_cocylindrical}, showing that, for any $i\in \Z$, $B(\al i, \al{k+i-1})$ is simple and $B(\al i, \al{k+i})$ is non-simple will prove our claim.

 Now, recall that $\al i = \eta^i (\al 0)$ (Theorem \ref{thm:infinite_homotopy_class}), so, 
\begin{equation*}
h \cdot \al i = h \cdot \eta^{i}(\al 0) = \eta^{i}\left(h \cdot \al 0 \right) \in \eta^{i}\left(L(\al{k-1}, \al k) \right) = L(\al{k+i-1}, \al{k+i}),
\end{equation*}
which implies that $B(\al i, \al{k+i})$ is non-simple. Clearly, this argument also implies that if $B(\al i, \al{k+i-1})$ was non-simple, then $B(\al 0, \al{k-1})$ would not be simple either.
\end{proof}


\section{Action of the fundamental group on $\univ$ and co-cylindrical orbits} \label{sec:action_of_pi1}

Thanks to Thurston's work in \cite{Thurston:3MFC}, we know that the fundamental group of a $3$-manifold 
admitting an $\R$-covered foliation acts on the universal circle 
of (say) the stable foliation. 
There is a remarkable link between the existence of co-cylindrical orbits and the action of $\pi_1(M)$ on \emph{pairs} of points in $\univ$.

Recall that $\phi$ is a skewed, $\R$-covered Anosov flow
with $\fs$ transversely orientable and $\univ$ is the universal
circle of $\fs$.
The reason we are interested in pairs of points in $\univ$ is that an orbit of $\hflot$ ``projects'' to two points on the universal circle. First, let us identify $\univ$ with $\leafs/ \eta^u\circ \eta^s$. The function $\eta^u$ induces an homeomorphism between $\leafu/ \eta^s\circ \eta^u$ and $\univ$ that preserves the action of $\pi_1(M)$.
 Let $\wt{\alpha}$ be an orbit of $\hflot$. The stable leaf $\hfs(\wt{\alpha})$ defines a point in $\univ$ by our identification, and $\hfu(\wt{\alpha})$ defines another point via $\eta^u$. Hence, an orbit defines two points on $\univ$. We can also see those two points as the projection to $\univ$ of the ideal points of the orbit $\wt{\alpha}$ on the boundary at infinity of the stable leaf $\hfs(\wt{\alpha})$.

\begin{defin}
 Let $(a^+, a^-)$ and $(b^+, b^-)$ be two pairs of points in $\univ$. We say that $(a^+, a^-)$ and $(b^+, b^-)$ are \emph{linked} if, for some order on $\univ$, we have
\begin{equation*}
 a^- < b^- <a^+ < b^+.
\end{equation*}
\end{defin}

\begin{prop}\label{prop:intersects_equivalent_simple}
 Let $\alpha$ be a periodic orbit of $\flot$, $\wt{\alpha}$ a lift to $\M$ and $(a^+, a^-)$ the projection of $\wt{\alpha}$ on $\univ$.
 The co-cylindrical class of $\alpha$ is finite if and only if there exists $h\in \pi_1(M)$ such that $(a^+, a^-)$ and $(h \cdot a^+, h \cdot a^-)$ are linked.
\end{prop}

\begin{proof}
 If the co-cylindrical class of $\alpha$ is finite, then (by Theorem \ref{thm:simple_and_cocylindrical}) the chain of lozenges containing $\wt{\alpha}$ is non-simple. So, there exists $h\in \pi_1(M)$ such that $h\cdot \wt{\alpha} \in L(\al i, \al{i+1})$. Hence, we have that $\hfs(\al i) < h\cdot \hfs(\wt{\alpha}) < \hfs(\al{i+1})$ and $\hfu(\al i) < h\cdot \hfu(\wt{\alpha}) < \hfu(\al{i+1})$. Projecting these leaves to $\univ$ shows that $(a^+, a^-)$ and $h \cdot (a^+, a^-)$ are linked.

Conversely, if there exists $h \in \pi_1(M)$ such that $(a^+, a^-)$ and $h \cdot (a^+, a^-)$ are linked, then $h\cdot \wt{\alpha} \in L(\al i, \al{i+1})$ for some $i$. Hence, by Theorem \ref{thm:simple_and_cocylindrical}, the co-cylindrical class of $\alpha$ must be finite.
\end{proof}

\begin{rem}
From now on in this section we assume that $M$ is atoroidal.
Thurston \cite{Thurston:3MFC} proved that there is a pseudo-Anosov flow
$\psi^t$ transverse to $\fs$ and regulating for $\fs$: 
every orbit of $\wt{\psi^t}$ intersects every leaf of $\fs$.
\end{rem}

\begin{thm} \label{thm:everything_intersects} 
 Let $(a^+, a^-)$ be the projection on $\univ$ of a periodic orbit $\wt{\alpha}$ of $\hflot$. Then, there exists $h\in \pi_1(M)$ such that $(a^+, a^-)$ and $(h \cdot a^+, h \cdot a^-)$ are linked.
\end{thm}

\begin{proof}
%
The universal circle $\univ$ can be seen as the boundary of the orbit space of 
the regulating pseudo-Anosov flow $\psi^t$ (\cite{Fenley:Ideal_boundaries}). In our situation,
this fact is easily seen to be true: the orbit space of $\wt{\psi^t}$ 
can be identified to a stable leaf $F$ of $\hfs$ (because $\psi^t$ is regulating for $\fs$), and the
universal circle is identified to the circle at infinity of $F$.

In this proof we will abuse notation and use the same notation
for the points $a^+$ in $\univ$ and the corresponding
point in the ideal circle of a leaf $F$ of
$\hfs$.
It was proved in \cite{Fenley:Ideal_boundaries} that
the ideal points $a^+, a^-$ cannot be ideal points of a leaf
of $\wt{\psi}^t$ in $F$.
Then, in the compactification of
the orbit space, 
$a^+$ and  $a^-$ have neighborhood systems in
$\orb(\wt{\psi^t}) \cup \univ$ defined
by unstable leaves.
Let $l^u$ be a non-singular unstable leaf of the pseudo-Anosov regulating flow separating $a^+$ and
$a^-$. The unstable leaf $l^u$ separates $\univ$ in two connected components $I^+$ and $I^-$,
containing respectively $a^+$ and $a^-$. There exists a singular stable leaf $l^s$ of $\wt{\psi}^t$,
intersecting $l^u$ and such that there are endpoints of $l^s$ in $I^+$ on both sides of $a^+$ in
$I^+$ (see Figure \ref{fig_first_case}). The existence of $l^s$ was proved in 
\cite{Fenley:Ideal_boundaries} and is related 
to the fact that stable  or unstable leaves of $\wt{\psi^t}$ define a 
neighborhood system of any point of $\univ$ in 
$\univ \cup \orb(\psi^t)$ (see \cite{Fenley:Ideal_boundaries}).

Let us fix an orientation on $\univ$ and write $x_1, \dots, x_p$ for the endpoints of $l^s$, chosen such that $a^+ \in (x_1,x_2)$. So there exists $2 \leq k \leq p$ such that $a^- \in (x_{k},x_{k+1})$ (with the convention that $x_{p+1}=x_1$).

As $M$ is atoroidal, the pseudo-Anosov flow $\psi^t$ is transitive (see \cite{Mosher}). Hence, the union of periodic orbits of $\psi^t$ is dense. So, there exists a non-singular orbit $A$ of $\wt{\psi}^t$, which is a lift of a periodic orbit of $\psi^t$, such that its unstable leaf $l^u(A)$ still separates $a^+$ and $a^-$ and such that its stable leaf $l^s(A)$ ends in $(x_{k},x_{k+1})$ on both sides of $a^-$. So $l^s(A)$ also separates $a^+$ and $a^-$. To get such an orbit $A$, it suffices to take $A$ close to the intersection of $l^s$ and $l^u$ in the section that has $a^-$ in its boundary (see Figure \ref{fig_first_case}).
This was proved in \cite{Fenley:Ideal_boundaries} and uses the fact 
that the leaf space of the stable foliation of $\psi^t$ lifted to $\wt{M}$
is Hausdorff $-$ also proved in 
\cite{Fenley:Ideal_boundaries}.

\begin{figure}[h]
\begin{pspicture}(0,-3.5)(6.5,3.3)
 

\pscircle[linewidth=0.04,dimen=outer](3.0,-0.32921875){3.0}
\psbezier[linewidth=0.04](0.4,-1.7292187)(1.0,-1.1292187)(2.0044713,-1.0236772)(2.6,-0.92921877)(3.1955287,-0.83476025)(4.8,-0.7292187)(5.8,-1.3292187)
\psbezier[linewidth=0.04](2.0,-3.1292188)(2.0,-2.1292188)(3.6,1.4707812)(4.4,2.2707813)
\psbezier[linewidth=0.04](3.0,-0.12921876)(2.0,0.27078125)(1.0,1.0707812)(1.0,1.8707813)
\psdots[dotsize=0.16](4.8,-2.7292187)
\psdots[dotsize=0.16](2.0,2.4707813)

\psbezier[linewidth=0.04,linecolor=red,linestyle=dashed,dash=0.16cm 0.16cm](2.6,-3.3292189)(2.4,-2.3292189)(3.2,1.0707812)(4.8,2.0707812)
\psbezier[linewidth=0.04,linecolor=red,linestyle=dashed,dash=0.16cm 0.16cm](0.6,-2.1292188)(1.0,-1.3292187)(4.8,-0.92921877)(5.6,-1.7292187)
\psdots[dotsize=0.16](2.8,-1.3292187)

\rput(0.018574925, 0.015528335){\psdots[dotsize=0.26,dotangle=28.739796,dotstyle=triangle*,linecolor=red](5.8,0.67078125)}
\rput(-0.0113531565, -0.0037936294){\psdots[dotsize=0.26,dotangle=-14.09651,dotstyle=triangle*,linecolor=red](4.0,-3.1292188)}
\rput(0.022998504, 0.06503281){\psdots[dotsize=0.26,dotangle=90.0,dotstyle=triangle*,linecolor=red](0.4,1.0707812)}
\psdots[dotsize=0.26,dotstyle=triangle*,linecolor=red](1.6,-2.9292188)

\psdots[dotsize=0.17,fillstyle=solid,dotstyle=o](3.2,-3.3)
\psdots[dotsize=0.18,fillstyle=solid,dotstyle=o](3.9,2.5)

\put(3,-1.6){$A$}
\put(5,-2.9){$a^-$}
\rput(1.97,2.8){$a^+$}

\put(3.9,2.7){$h\cdot a^+$}
\put(3,-3.7){$h\cdot a^-$}

\rput(5,-0.8){$l^u$}
\rput(3,0.4){$l^s$}
\rput(4.7,2.4){$x_1$}
\rput(0.8,2){$x_2$}
\rput(2,-3.4){$x_3$}
\end{pspicture} 
\caption{} \label{fig_first_case}
\end{figure}

The leaves $l^s(A)$ and $l^u(A)$ determine four quadrants and, by our construction, $a^+$ and $a^-$ are in two opposite quadrants. So if $h\in \pi_1(M)$ is a non-trivial element stabilizing $A$ and the ideal points of $l^s(A)$ and $l^u(A)$, the action of $h$ on $\univ$ moves $a^+$ and $a^-$ in opposite directions. Hence, $(a^+, a^-)$ and $h \cdot( a^+,  a^-)$ are linked.
\end{proof}

\begin{rem}\label{rem:intersects}
 Suppose that there exists $h\in \pi_1(M)$ such that $(a^+, a^-)$ and the
image $(h \cdot a^+, h \cdot a^-)$ are linked. Denote by $(\al i)$ the orbits in $\M$ projecting to $(a^+, a^-)$ and $\alpha_i = \pi (\al i)$ their projection to $M$. Then, for any $i$, there exist a $j$ and a $t$ such that $\psi^t(\alpha_i) \cap \alpha_j \neq \emptyset$. So if we push one orbit by the regulating flow, we obtain an actual intersection.
 \end{rem}

If we consider the geodesic flow case now, there is also a natural circle at infinity. Just take the visual boundary $\wt{\Sigma}(\infty)$ and the fundamental group $\pi_1(S\Sigma)$ naturally acts on it. Suppose that there exists $h\in \pi_1(M)$ such that $(a^+, a^-)$ and $(h \cdot a^+, h \cdot a^-)$ are linked. This means that the only geodesic in $\Sigma$ such that a lift of it has endpoints $(a^+, a^-)$ is non-simple. Hence, as there always is simple closed geodesics, the geodesic flow case is once again in sharp contrast with the atoroidal case we studied here.

As a corollary of Proposition \ref{prop:intersects_equivalent_simple} and Theorem \ref{thm:everything_intersects}, we obtain:
\begin{thm}
 Every co-cylindrical class is finite.
\end{thm}

Note that it is still an open question whether a co-cylindrical class can be non-trivial. We only know that some are:
\begin{prop}
 There exist periodic orbits of $\flot$ with trivial co-cylindrical class. 
\end{prop}

\begin{rem}
For such an orbit, Proposition \ref{prop:cardinality_isotopy_class} shows that every other orbit in the double free homotopy class must also have a trivial co-cylindrical class.
\end{rem}

\begin{proof}
 Let $V$ be a flow box of $\flot$, as $\flot$ is transitive, we can pick a long segment of a dense orbit that $\eps$-fills $V$. Then, by the Anosov Closing lemma (see \cite{KatokHassel}), we get a periodic orbit $\alpha$ that $2\eps$-fills $V$. Now, choose $x$ on one of the connected components of $\alpha \cap V$. If $\eps$ was chosen small enough, then there must exist $y$ on another connected component of $\alpha \cap V$ such that there is a close path $c$ \emph{staying in $V$}, starting at $x$ going through the positive stable leaf of $x$, then the negative unstable leaf of $y$, then the negative stable leaf of $y$ and finally close up along the positive stable leaf of $x$. If we lift the path $c$ to the universal cover of $M$ and project it to the orbit space $\orb$, as $V$ has no topology, we see that the projection of the lift of $y$ must be inside the lozenge determined by the lift of $x$ (remember that we chose our flow so that the lozenges orientation is $(+,+,-,-)$, otherwise, we would have to 
modify our path $c$, see Figure \ref{fig:lozenge++--}). Hence the lozenge is non-simple and therefore the co-cylindrical class of $\alpha$ is trivial.
\end{proof}

\bibliographystyle{amsplaineprint}
\bibliography{tout}
\end{document}